\documentclass[12pt]{article}
\usepackage{mathtools}
\usepackage{amsmath}
\usepackage{amsfonts}
\usepackage{latexsym}
\usepackage{amssymb}
\usepackage{stmaryrd}
\usepackage{overpic}
\usepackage{graphicx}
\newcommand{\safeincludegraphics}[2][]{%
  \IfFileExists{#2}{\includegraphics[#1]{#2}}{%
    \fbox{\parbox[c][0.20\textheight][c]{0.78\linewidth}{\centering
    Placeholder for Figure\\[0.5ex]\texttt{\detokenize{#2}}}}%
  }%
}
\usepackage{multirow}
\usepackage{mathrsfs} 
\usepackage{float}  
\usepackage{color}
\usepackage{subcaption}
\usepackage{booktabs}
\newtheorem{theorem}{Theorem}[section]

\newtheorem{proposition}[theorem]{Proposition}
\newtheorem{remark}[theorem]{Remark}
\newenvironment{proof}[1][Proof]{\textbf{#1.} }
{\ \rule{0.75em}{0.75em}\smallskip}

\begin{document}

\begin{center}
\large\bf {Fourier Neural Operators with Least-Squares Readout Refit for Learning Random Obstacle-to-Solution Maps}
\end{center}

\begin{center}
Chenhui Zhu\footnote{School of Mathematics and Statistics, Xi'an Jiaotong University,
Xi'an, Shaanxi 710049, P. R. China. Email: {chzhu@stu.xjtu.edu.cn}}\quad and \quad
Fei Wang\footnote{School of Mathematics and Statistics \& State Key Laboratory of Multiphase Flow in Power Engineering, Xi'an Jiaotong University, Xi'an, Shaanxi 710049, China. The work of this author was partially supported by National Key R\&D Program of China (2025YFA1016400) and National Natural Science Foundation of China (Grant No.\ 92470115). Email: {\tt feiwang.xjtu@xjtu.edu.cn}}

\end{center}
\bigskip
\begin{quote}

{\bf Abstract.}
We study operator learning for random obstacle-to-solution maps arising from elliptic variational inequalities with finite-band self-affine random obstacle fields. Instead of introducing an explicit truncated stochastic parametrization of the random input, we learn the map directly from sampled obstacle realizations on a fixed grid. This problem is challenging because the solution is governed not only by the obstacle field itself, but also by the induced contact set and free-boundary geometry. We introduce a post-training least-squares readout refit for the Fourier neural operator (FNO). After the FNO is trained end to end, its nonlinear backbone is frozen and the final affine readout is recomputed by solving the induced linear least-squares problem over all training samples and grid points. The refit yields the empirical squared-error optimal readout for the learned frozen features while leaving the nonlinear representation unchanged. We compare vanilla DeepONet, POD-DeepONet, a two-stage DeepONet baseline, FNO, and FNO with least-squares readout refit (FNO-LS) on two obstacle ensembles with different amplitude levels. Numerical results show that FNO-LS achieves the strongest overall performance among the tested models, particularly for higher-amplitude obstacles with more complex contact geometry. The method improves average field accuracy, contact-set recovery, and obstacle-violation metrics at low additional cost, especially when the FNO backbone is informative but not fully converged. These results suggest that least-squares readout refit is a simple and effective post-training enhancement for learning random obstacle-to-solution maps.


\end{quote}

\section{Introduction}

Obstacle problems are prototypical elliptic variational inequalities and free-boundary problems. They arise in contact mechanics, lubrication, elastostatics, optimal stopping, and many other settings in which unilateral constraints are intrinsic to the model. Their distinctive difficulty is that the contact set (coincidence set), and hence the free boundary, is not known a priori and must be determined together with the solution. This makes both analysis and computation substantially more delicate than for standard elliptic boundary value problems (\cite{DL1976,Ro1987,Kinderlehrer2000,TLG1981}).

In many applications, obstacle problems must be solved repeatedly for varying inputs. Uncertainty may enter through coefficients, forcing terms, boundary data, or the obstacle itself, leading naturally to stochastic obstacle problems and, more broadly, stochastic variational inequalities. Classical approaches include stochastic Galerkin methods (\cite{Babuska2002,Babuska2004}), stochastic collocation \cite{Babuska2007}, polynomial-chaos approximations (\cite{Xiu2007,Forster2010}), and multilevel Monte Carlo finite element methods (\cite{Kornhuber2014}). A common strategy is to first represent the random input by a finite-dimensional parametrization, for example through a truncated Karhunen--Lo\`eve expansion, and then solve the resulting deterministic-parametric family (\cite{Schwab2006,Todor2007}). This route is mathematically natural, but it becomes increasingly expensive when the random field is rough or strongly multiscale, since many modes may be needed to resolve the input accurately.

Existing work relevant to this paper can be grouped into four strands. First, there is a large classical literature on deterministic obstacle problems and variational inequalities, including finite element (\cite{Glowinski1984,Hlavacek1988}), discontinuous Galerkin (\cite{wang2010discontinuous, wang2014discontinuous}), and virtual element methods (\cite{wang2018friction,wang2020obstacle}). Second, stochastic obstacle problems and stochastic variational inequalities have been studied using stochastic Galerkin (\cite{Zhu2026Obstacle}), polynomial-chaos approximations (\cite{Forster2010}) and Monte Carlo techniques (\cite{Kornhuber2014,Bierig2015}), with limited regularity of the random input remaining a central challenge. Third, on the scientific machine learning side, several neural-network methods have been proposed as single-instance solvers for obstacle-type problems, including finite-element and quadratic-programming based neural approaches (\cite{Darehmiraki2022Obstacle}), energy minimization methods (\cite{Zhao2022Obstacle}), penalty-based formulations (\cite{Cheng2023ObstacleDNN}), residual-based PINN approaches (\cite{ElBahja2025PINNObstacle}), weak adversarial minmax methods (\cite{Alphonse2024EVI}), and proximal PINN formulations (\cite{Gao2025ProxPINNs}). Fourth, at the operator-learning level, solution operators for variational inequalities have been studied using proximal neural-network constructions (\cite{SchwabStein2022ProxNet}), and recent work has begun to address high-dimensional parametric obstacle problems with specialized multilevel architectures (\cite{EigelHeissSchutte2025}). Direct operator learning from rough random obstacle fields, however, remains much less explored than either single-instance neural solvers or smoother parametric operator-learning settings.

Compared with standard elliptic solution operators, the obstacle-to-solution map contains an additional geometric difficulty: small changes in the obstacle can move the contact set, create or remove connected components, or shift thin free-boundary branches. Consequently, a surrogate may achieve a small global field error while still misidentifying the contact region or violating the unilateral constraint. This motivates both the use of grid-based neural operators and the structure-sensitive evaluation metrics adopted in this work.

Among neural-operator architectures, DeepONet (\cite{Lu2021DeepONet}) and the FNO are two of the most widely used models (\cite{Li2021FNO,Kovachki2023NeuralOperator}). DeepONet is based on a branch--trunk decomposition and is flexible across a wide range of operator-learning settings. FNO, by contrast, propagates information through spectral convolution and is often particularly effective for grid-based problems with long-range interactions and multiscale structure. More broadly, recent theory shows that generic operator learning may suffer from a curse of parametric complexity unless additional structure is exploited (\cite{LanthalerStuart2026}); this provides further motivation for studying problem-adapted architectures on specific classes of variational inequalities. Because the rough random obstacles considered here contain oscillations across many spatial scales, FNO is a natural candidate for the present problem.

In this paper, we propose FNO-LS, an FNO-based operator learning method with a post-training least-squares readout refit for random obstacle-to-solution maps. We formulate the task as direct fixed-grid field-to-field learning: the input is a sampled random obstacle field, and the output is the corresponding variational-inequality solution on the same grid. After an FNO is trained end to end, we freeze its nonlinear backbone and recompute only the final affine readout by solving the induced linear least-squares problem over all training samples and grid points. This yields the empirical squared-error optimal readout for the learned frozen features at low additional cost, while leaving the nonlinear representation unchanged.

We compare the resulting FNO-LS method with vanilla DeepONet, POD-DeepONet, a two-stage DeepONet baseline (\cite{Lee2024}), and standard FNO on two obstacle ensembles with different amplitude levels. The experiments show that FNO-LS gives the strongest overall performance among the tested models, especially in the higher-amplitude regime where the contact geometry is more intricate. To reflect the variational-inequality structure of the problem, we report not only relative \(L^2\) field error but also contact-set intersection-over-union and obstacle-violation metrics, which assess contact-region recovery and satisfaction of the unilateral constraint.

The rest of the paper is organized as follows. In Section~2, we review the deterministic and random obstacle problems and introduce the random obstacle model used in the experiments. In Section~3, we present the neural-operator viewpoint together with the proposed least-squares readout refit strategy. In Section~4, we report the numerical results and discussion.

\section{Deterministic and random obstacle problems}\label{sec2}
\setcounter{equation}0

\subsection{The classical obstacle problem}

Obstacle problems are prototypical elliptic variational inequalities and free-boundary problems (\cite{Ro1987,Atkinson2009Theoretical}). A standard mechanical interpretation is that of an elastic membrane constrained to stay above a rigid obstacle. Let $D\subset \mathbb{R}^n$ be an open, bounded Lipschitz domain, let $f$ be a forcing term, let $a$ be a diffusion coefficient, and let $g$ be the obstacle function. The solution $u$ is required to satisfy the inequality constraint $u\ge g$ and the governing elliptic equation away from the contact set.

The weak formulation reads as follows: find $u\in K$ such that
\begin{align}
\label{variation_inequality}
\int_D a\nabla u\cdot \nabla(v-u)\,\mathrm{d}x
\ge
\int_D f(v-u)\,\mathrm{d}x
\qquad \forall v\in K,
\end{align}
where $f\in L^2(D)$, $a\in L^\infty(D)$, and
\[
K:=\{v\in H_0^1(D)\,:\, v\ge g \ \text{a.e. in } D\}.
\]
Here $g$ denotes the obstacle height.  Under the standard assumptions $a(x)\ge a_0>0$ almost everywhere, $f\in L^2(D)$, and $g\in H^1(D)$ with $g\le 0$ on $\partial D$, problem \eqref{variation_inequality} admits a unique weak solution; see \cite[page 3]{Glowinski1984} and \cite[Theorem 11.3.9]{Atkinson2009Theoretical}.

A distinctive feature of the obstacle problem is that the contact set
\[
\mathcal{C}:=\{x\in D:\ u(x)=g(x)\}
\]
is not known in advance and must be determined together with the solution. This free-boundary structure is one of the main sources of computational difficulty.

For notational simplicity, we present the obstacle problem with homogeneous Dirichlet boundary conditions. More generally, one may prescribe nonhomogeneous Dirichlet data $u_D \in H^{1/2}(\partial D)$ and consider the admissible set
\[
K_{g,u_D}:=\{v\in H^1(D): \operatorname{Tr}(v)=u_D,\ v\ge g \ \text{a.e. in }D\}.
\]
Introducing a lifting function $\ell\in H^1(D)$ with $\operatorname{Tr}(\ell)=u_D$ and setting $\widetilde u=u-\ell$, the problem is reduced to a variational inequality on a closed convex subset of $H_0^1(D)$ with the same bilinear form and a modified linear functional. Hence, under the same assumptions on the domain, coefficient, source term, and obstacle, together with the nonemptiness of the admissible set, the standard existence and uniqueness theory still applies (\cite{Glowinski1984,Atkinson2009Theoretical,Ro1987,Kinderlehrer2000}).

\subsection{The random obstacle problem}

In many applications, the coefficient \(a\), the source term \(f\), and the obstacle function \(g\) are not known exactly and should be modeled as random fields.

To formulate the random obstacle problem, we briefly introduce the underlying functional setting. Let \((\Omega,\mathcal{F},\mathbb{P})\) be a probability space, where \(\Omega\) denotes the set of elementary events, \(\mathcal{F}\) is a \(\sigma\)-algebra, and \(\mathbb{P}\) is a probability measure. Let \(D\subset\mathbb{R}^n\) be a bounded Lipschitz domain. If \(X\) is a Banach space of real-valued functions defined on \(D\), then \(L^p(\Omega;X)\), \(1\le p<\infty\), denotes the Bochner space of all \(X\)-valued random variables \(v\) such that
\[
\|v\|_{L^p(\Omega;X)}^p
:=
\int_\Omega \|v(\cdot,\omega)\|_X^p\,\mathrm{d}\mathbb{P}(\omega)
<\infty.
\]
For \(p=\infty\), the norm is defined by
\[
\|v\|_{L^\infty(\Omega;X)}
:=
\operatorname*{ess\,sup}_{\omega\in\Omega}\|v(\cdot,\omega)\|_X.
\]

We now consider the random obstacle problem (\cite{Todor2007,Forster2010,Zhu2026Obstacle}): Given \(a \in L^{\infty}\left(\Omega; L^{\infty}(D)\right)\), \( f \in L^2(\Omega; L^2(D))\), \(g \in L^{2}\left(\Omega; H^{1}(D)\right)\), find a function $u$ such that for almost every \( \omega \in \Omega \), \( u(\cdot, \omega) \in K_{g(\omega)} \) and the following inequality holds:
{\small \begin{equation}
	\label{eq2}
	\int_{D} a(x, \omega) \nabla u({ x}, \omega) \cdot \nabla (v({ x})-u({x}, \omega)) \, \mathrm{d}x \geq \int_{D} f(x, \omega)(v({x})-u({ x}, \omega)) \, \mathrm{d} x \quad\forall\,v \in K_{g(\omega)},
\end{equation}} where \(K_{g(\omega)} = \{v \in H_0^1(D) \mid v \geq g(\cdot, \omega) \text{ a.e. in } D\}\). {The gradient operator \(\nabla\) refers to differentiation with respect to $x \in D$ unless otherwise stated.} When sufficient regularity is available, the corresponding pointwise formulation is to find a random function $u(x,\omega)$ such that (cf.\ \cite[Example 11.1.1]{Atkinson2009Theoretical})
\begin{equation}
    \label{obstacle_inequality}
\begin{cases}
-\nabla\cdot\bigl(a(x,\omega)\nabla u(x,\omega)\bigr)\ge f(x,\omega), & x\in D,\\[1mm]
u(x,\omega)\ge g(x,\omega), & x\in D,\\[1mm]
\bigl(-\nabla\cdot(a(x,\omega)\nabla u(x,\omega))-f(x,\omega)\bigr)
\bigl(u(x,\omega)-g(x,\omega)\bigr)=0, & x\in D,\\[1mm]
u(x,\omega)=0, & x\in \partial D.
\end{cases}
\end{equation}

As in the deterministic case, the contact set
\[
\mathcal{C}_{g(\omega)}:=\{x\in D:\ u(x,\omega)=g(x,\omega)\}
\]
is unknown.

{
If there exist constants \(0<a_{\min}\le a_{\max}<\infty\) such that
\[
a_{\min}\le a(x,\omega)\le a_{\max}
\quad\text{for a.e. }(x,\omega)\in D\times\Omega,
\]
and if \(f(\cdot,\omega)\in L^2(D)\), \(g(\cdot,\omega)\in H^1(D)\), and
\(K_{g(\omega),u_D}\neq \varnothing\) for almost every \(\omega\in\Omega\), then for almost every \(\omega\in\Omega\) the random variational inequality has a unique weak solution.
}

From the operator-learning viewpoint, it is natural to regard the solution as a map defined on an admissible class of obstacle fields. If \(a\) and \(f\) are fixed, we write
\begin{equation}\label{eq:parameter_to_solution_map}
\mathcal{G}_{\mathrm{obs}}:\mathcal{X}_{\mathrm{ad}}\to H^1(D),\qquad
g\mapsto u(g),
\end{equation}
where \(\mathcal{X}_{\mathrm{ad}}\) denotes a class of admissible obstacles satisfying the appropriate boundary compatibility condition. If the coefficient, source term, and obstacle all vary, one may instead consider the more general solution map
\[
\mathcal{G}:(a,f,g)\mapsto u.
\]
In this paper, \(a\) and \(f\) are fixed and only the obstacle varies. Thus the operator-learning task is the obstacle-to-solution map \(g\mapsto u\).

We emphasize that, unlike many classical approaches to stochastic PDEs and stochastic variational inequalities, we do not introduce a KL expansion of the random obstacle field \(g\) and then truncate it to obtain a finite-dimensional parametric problem. In the traditional framework, one first represents the random input by a finite number of random variables through a truncated KL expansion, and then applies stochastic Galerkin, stochastic collocation, or related high-dimensional approximation techniques to the resulting deterministic-parametric problem (\cite{Schwab2006,Babuska2002,Babuska2004,Babuska2007,Todor2007}). Although this strategy is mathematically natural, it generally introduces an additional modeling error through truncation of the random field, and the resulting parametric problem may become high-dimensional when many KL modes are needed to resolve rough or strongly oscillatory inputs. This substantially increases the computational burden and may lead to the curse of dimensionality. In contrast, the present work adopts a direct operator-learning viewpoint: we learn the map from the random obstacle field \(g\) to the corresponding solution \(u\) directly from sampled field realizations, without constructing an explicit finite-dimensional stochastic parametrization of \(g\). This avoids introducing an additional KL truncation of the sampled obstacle field at the learning stage.

Following the rough-surface model in~\cite{Bierig2015} and the power-spectrum description in~\cite{Persson2006}, we generate the finite-band self-affine random obstacle ensemble. Specifically, the obstacle is generated by the finite Fourier expansion
\begin{equation}\label{eq:obstacle_function}
g(x)=\sum_{q\in Q} B_q(H;\alpha)\cos\bigl(q\cdot x+\phi_q\bigr),
\qquad x\in D,
\end{equation}
where the sum is taken over the wave-vector set
\[
Q:=\{q\in(\pi\mathbb{Z})^2:\ 1\le \|q\|_2\le 26\}.
\]
The phase shifts \(\{\phi_q\}_{q\in Q}\) are independent random variables uniformly distributed on \([0,2\pi]\). Thus the obstacle is represented as a finite superposition of Fourier modes with random phases, with modal amplitudes controlled by the Hurst exponent \(H\) and the scale parameter \(\alpha\).

The amplitudes are chosen according to the standard self-affine roughness scaling. In such models, the surface roughness power spectrum satisfies
\begin{equation}\label{eq:self-affine-spectrum}
C(q)\sim \|q\|_2^{-2(H+1)},
\end{equation}
where \(H\in[0,1]\) is the Hurst exponent. Since Fourier amplitudes scale like the square root of the power spectrum, smaller values of \(H\) assign relatively more weight to high-frequency modes and therefore produce rougher obstacle profiles.

In the experiments, the modal amplitudes are chosen as
\begin{equation}\label{eq:Bq-general}
B_q(H;\alpha)
=
\alpha\bigl(2\pi\max(\|q\|_2,10)\bigr)^{-(H+1)},
\qquad q\in Q,
\end{equation}
and set \(B_q(H;\alpha)=0\) for \(q\notin Q\). The parameter \(\alpha>0\) controls the overall scale of the obstacle. The random variables are sampled as
\begin{equation}\label{eq:randomness}
H\sim\mathcal{U}(0,1),
\qquad
\phi_q\sim\mathcal{U}(0,2\pi),
\end{equation}
with \(H\) and all phase variables mutually independent.

Since the wave-vector set \(Q\) is finite, every sampled obstacle realization is a smooth trigonometric polynomial in \(x\). The computational difficulty addressed here therefore does not stem from low Sobolev regularity of the obstacle, but rather from its oscillatory finite-band geometry and the resulting sample-dependent contact set. 

\section{Neural operator architectures for random obstacle problems}
\setcounter{equation}0
\subsection{Deep Operator Networks}

DeepONet is a neural-operator architecture for approximating nonlinear operators between function spaces (\cite{Lu2021DeepONet}). Let
\[
\mathcal{G}:\mathcal{U}\to\mathcal{V}
\]
be an operator mapping an input function $g\in\mathcal{U}$ to an output function $\mathcal{G}(g)\in\mathcal{V}$. In DeepONet, the value of the output at a query point $y$ is represented by combining two subnetworks: a branch network that encodes the input function and a trunk network that encodes the evaluation coordinate.

For computation, the input function is sampled at sensor locations $\{x_1,x_2,\dots,x_m\}$ and represented by
\[
\bigl(g(x_1),g(x_2),\dots,g(x_m)\bigr)\in\mathbb{R}^m.
\]
This vector is passed to the branch network. The trunk network takes the coordinate $y$ as input. If the branch and trunk outputs are denoted by $b(g)\in\mathbb{R}^p$ and $t(y)\in\mathbb{R}^p$, respectively, then the DeepONet approximation has the form
\begin{equation}\label{eq:deeponet_general}
\mathcal{G}_{\theta}(g)(y)=\sum_{i=1}^{p} b_i(g)\,t_i(y) + b_0,
\end{equation} where $\theta$ collects the trainable parameters. Thus, DeepONet represents the solution operator through a data-dependent coefficient vector, a coordinate-dependent basis, and a scalar bias term.

\begin{figure}[htbp]
    \centering
    \safeincludegraphics[width=1\linewidth, page=1]{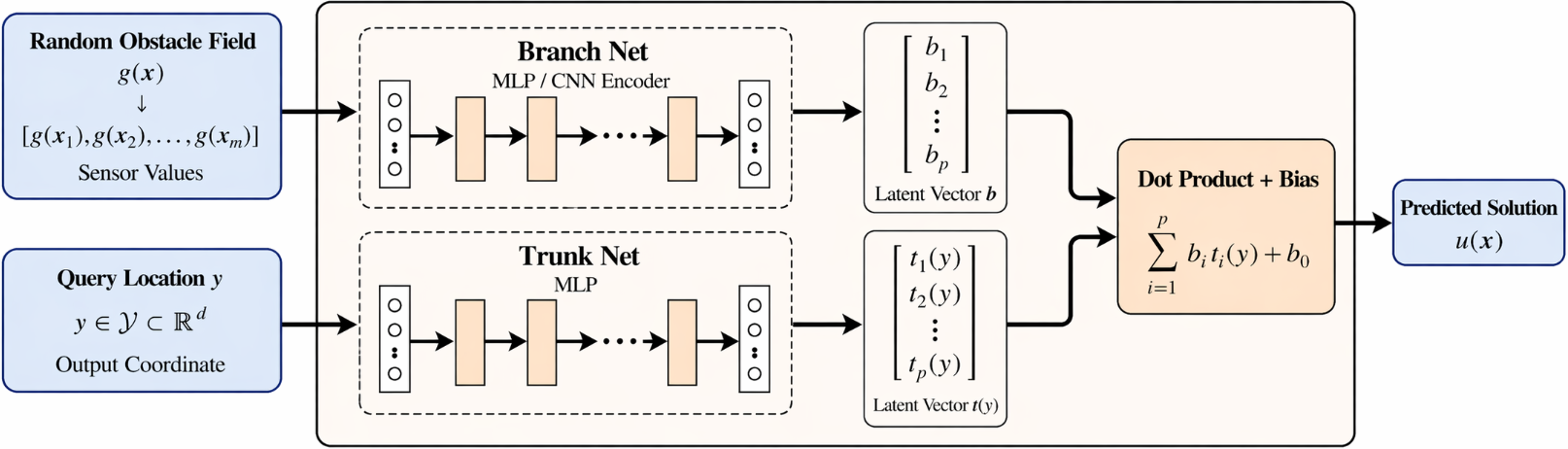}
    \caption{Schematic of the DeepONet architecture. The branch network encodes the sampled input function, while the trunk network encodes the query coordinate. Their outputs are combined to produce the predicted operator value at the query point.}
    \label{fig:deeponet_architecture}
\end{figure}

Given training data
\[
\left\{ \left(g^{(n)}, y_j^{(n)}, u_j^{(n)}\right) \right\}_{n=1,j=1}^{N,q},
\]
where
\[
g^{(n)}=\left(g^{(n)}(x_1),g^{(n)}(x_2),\ldots,g^{(n)}(x_m)\right), 
\qquad
u_j^{(n)} = \mathcal{G}\!\left(g^{(n)}\right)\!\left(y_j^{(n)}\right),
\]
the parameters are typically learned by minimizing
\[
L(\theta)
=
\frac{1}{Nq}
\sum_{n=1}^N \sum_{j=1}^q
\left|
\mathcal{G}_\theta\!\left(g^{(n)}\right)\!\left(y_j^{(n)}\right)-u_j^{(n)}
\right|^2.
\]

\subsection{Fourier Neural Operator}

Unlike DeepONet, the Fourier neural operator works directly with grid-based representations of both the input and the output (\cite{Li2021FNO}). Let $\{x_j\}_{j=1}^{M}\subset D$ be a uniform grid. The input and output fields are represented by
\[
v=\bigl(v(x_1),v(x_2),\dots,v(x_M)\bigr),
\qquad
u=\bigl(u(x_1),u(x_2),\dots,u(x_M)\bigr).
\]
FNO learns a map from the discretized input field to the discretized output field on the same mesh.

The architecture consists of a lifting operator, a sequence of Fourier layers, and a projection operator. The lifting map sends the input to a higher-dimensional feature space,
\[
v_0(x)=\mathcal{P}(v)(x).
\]
Each Fourier layer then applies a spectral convolution together with a local linear transformation:
\begin{equation}\label{eq:fno_layer}
v_{\ell+1}(x)
=
\sigma\!\left(
W_\ell v_\ell(x)
+
\mathcal{F}^{-1}\!\left(
R_\ell(k)\cdot \mathcal{F}(v_\ell)(k)
\right)(x)
+b_\ell
\right),
\qquad \ell=0,\dots,L-1,
\end{equation}
where $\mathcal{F}$ and $\mathcal{F}^{-1}$ are the Fourier and inverse Fourier transforms, $R_\ell(k)$ is a learnable Fourier multiplier, $W_\ell$ is a pointwise linear map, and $\sigma$ is the activation function. In practice, only a finite number of Fourier modes are retained. After $L$ Fourier layers, the final projection gives
\[
\mathcal{G}_{\theta}(v)=\mathcal{Q}(v_L).
\]
Equivalently,
\begin{equation}\label{eq:fno_general}
\mathcal{G}_{\theta}(v)
=
\mathcal{Q}\circ \mathcal{L}_{L}\circ \mathcal{L}_{L-1}\circ \cdots \circ \mathcal{L}_{1}\circ \mathcal{P}(v),
\end{equation}
where $\mathcal{L}_{\ell}$ denotes the $\ell$-th Fourier layer.

\begin{figure}[htbp]
    \centering
    \safeincludegraphics[width=1\linewidth, page=1]{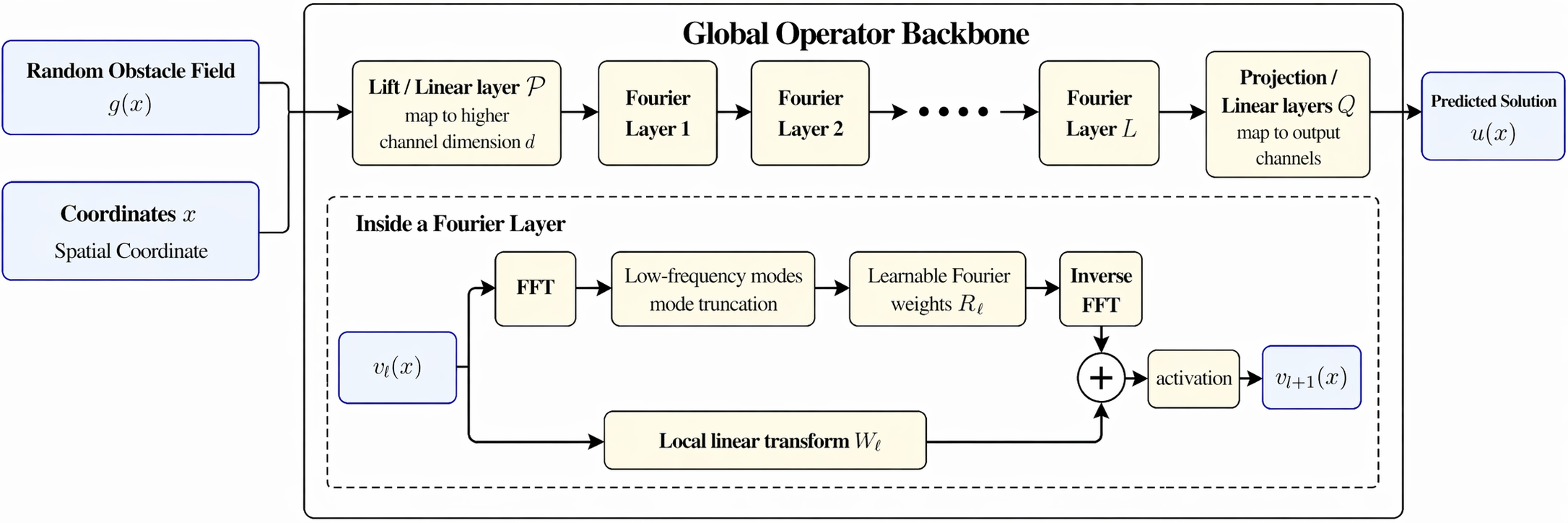}
    \caption{Schematic of the Fourier neural operator architecture. The lifting operator $\mathcal{P}$ maps the input field to a latent feature space, the Fourier layers $\mathcal{L}_\ell$ perform spectral convolution and local transformation, and the projection operator $\mathcal{Q}$ maps the final latent representation to the output field.}
    \label{fig:fno_architecture}
\end{figure}

Given training pairs
\[
\left\{\left(v^{(n)},u^{(n)}\right)\right\}_{n=1}^{N},
\]
the parameters are learned by minimizing
\begin{equation}\label{eq:fno_loss}
\mathcal{L}(\theta)
=
\frac{1}{NM}
\sum_{n=1}^{N}\sum_{j=1}^{M}
\left|
\mathcal{G}_{\theta}\bigl(v^{(n)}\bigr)(x_j)
-
u^{(n)}(x_j)
\right|^2.
\end{equation}
Because the spectral convolution acts globally on the grid, FNO is well suited to problems with long-range interactions and multiscale structure.

\subsection{FNO-LS: FNO with least-squares readout refit}
After end-to-end training, we write the learned FNO parameters as $(\psi^\star,\theta_0)$, where $\psi^\star$ collects all trained parameters except those of the final affine readout, and $\theta_0$ denotes the corresponding readout parameters. In the architecture used here, the projection head ends with a pointwise affine map. Hence, once $\psi^\star$ is frozen, the only remaining trainable component is this final affine readout.

Fix the spatial grid $\{x_j\}_{j=1}^M$. For each input obstacle field $g$, let
\[
z_\star(g)(x_j)\in\mathbb{R}^{d_r},\qquad j=1,\dots,M,
\]
denote the latent feature vector at $x_j$ produced by the frozen backbone $\psi^\star$, that is, by all layers preceding the last affine map. Thus the frozen backbone induces a fixed feature map $g\mapsto z_\star(g)$.

With the backbone frozen, the prediction at grid point $x_j$ can be written as
\begin{equation}\label{eq:last_linear_readout_revised}
\widehat u_\theta(g)(x_j)=w^\top z_\star(g)(x_j)+b,
\end{equation}
where $w\in\mathbb{R}^{d_r}$, $b\in\mathbb{R}$, and $\theta=[w^\top,b]^\top\in\mathbb{R}^p$
with $p=d_r+1$. Equivalently, if we define the augmented feature vector
\[
\widetilde z_\star(g)(x_j):=\bigl[z_\star(g)(x_j)^\top,1\bigr]^\top\in\mathbb{R}^p,
\]
then
\[
\widehat u_\theta(g)(x_j)=\theta^\top \widetilde z_\star(g)(x_j).
\]
Therefore, after freezing $\psi^\star$, Stage~2 searches only over the finite-dimensional affine readout class supported by the learned features.

The procedure is naturally separated into two stages, as illustrated in Fig.~\ref{fig:FNO_2st}:
\begin{itemize}
\item Stage 1: train the full FNO end-to-end, jointly optimizing the nonlinear backbone and the final affine readout;
\item Stage 2: freeze the backbone $\psi^\star$ and refit only the final affine readout by least squares, optionally with ridge regularization.
\end{itemize}

\begin{figure}[htbp]
    \centering
    \safeincludegraphics[width=1\linewidth, page=1]{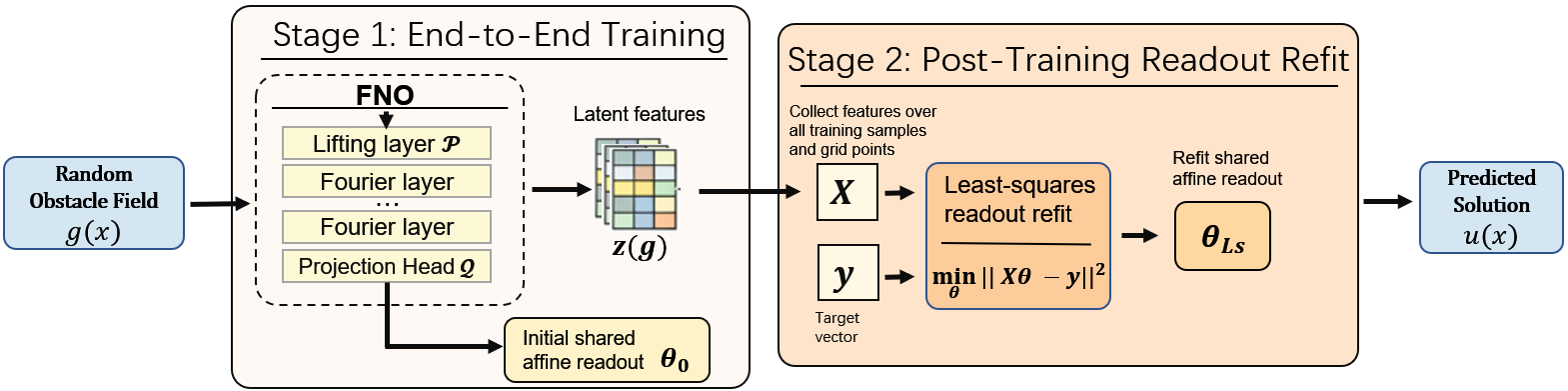}
    \caption{Two-stage FNO training and post-training readout refit. In Stage~1, the full network is trained end-to-end. In Stage~2, the nonlinear backbone is frozen and the final affine readout is refit by least squares.}
    \label{fig:FNO_2st}
\end{figure}

We now write Stage~2 precisely. Let
$\{(g^{(n)},u^{(n)})\}_{n=1}^{N_{\mathrm{train}}}$ be the training set. For each training sample
$n$ and grid point $x_j$, define
\[
\widetilde z_{n,j}:=\widetilde z_\star(g^{(n)})(x_j)\in\mathbb{R}^p,
\qquad
y_{n,j}:=u^{(n)}(x_j).
\]
Stacking all pairs $(n,j)$ into a single row index yields a design matrix
$X\in\mathbb{R}^{N_{\mathrm{row}}\times p}$ and a target vector
$y\in\mathbb{R}^{N_{\mathrm{row}}}$, where
\[
N_{\mathrm{row}}:=N_{\mathrm{train}}M.
\]
The unregularized readout refit is
\begin{equation}\label{eq:ls-objective-revised}
\widehat\theta_{\mathrm{LS}}
\in
\arg\min_{\theta\in\mathbb{R}^p}\|X\theta-y\|_2^2.
\end{equation}
Equivalently, $\widehat\theta_{\mathrm{LS}}$ minimizes the frozen-feature empirical squared loss
\[
\mathcal L_{\mathrm{fr}}(\theta):=
\frac{1}{N_{\mathrm{row}}}\|X\theta-y\|_2^2.
\]
If desired, one may instead solve the ridge-regularized problem
\begin{equation}\label{eq:ridge-objective-revised}
\widehat\theta_\lambda
\in
\arg\min_{\theta\in\mathbb{R}^p}
\Bigl(\|X\theta-y\|_2^2+\lambda\|\theta\|_2^2\Bigr),
\qquad \lambda\ge 0.
\end{equation}
In the numerical experiments below, however, we use the unregularized least-squares readout refit.

The next theorem records the basic empirical optimality property of the Stage~2 readout refit.
\begin{theorem}\label{thm:empirical-monotonicity}
Fix an end-to-end trained FNO with parameters $(\psi^\star,\theta_0)$, and keep the backbone parameters $\psi^\star$ fixed. Let $\widehat\theta_{\mathrm{LS}}$ be any minimizer of
\eqref{eq:ls-objective-revised}. Then
\[
\mathcal L_{\mathrm{fr}}(\widehat\theta_{\mathrm{LS}})
=
\min_{\theta\in\mathbb{R}^p}\mathcal L_{\mathrm{fr}}(\theta)
\le
\mathcal L_{\mathrm{fr}}(\theta_0).
\]
Equivalently,
\begin{equation}\label{eq:train-monotonicity}
\|X\widehat\theta_{\mathrm{LS}}-y\|_2^2
\le
\|X\theta_0-y\|_2^2.
\end{equation}
Therefore, for the same frozen features and the same training set, the least-squares readout refit cannot increase the empirical squared loss over the frozen-feature affine readout class defined by \eqref{eq:last_linear_readout_revised}.
\end{theorem}

\begin{proof}
The end-to-end readout $\theta_0$ is feasible for \eqref{eq:ls-objective-revised}, because it belongs to the same frozen-feature affine class. Since $\widehat\theta_{\mathrm{LS}}$ is a global minimizer of that objective, its value cannot exceed the value attained by $\theta_0$.
\end{proof}

Although elementary, Theorem~\ref{thm:empirical-monotonicity} is useful because it isolates exactly what Stage~2 accomplishes: for fixed features, least squares computes a full-data empirical minimizer over the same affine readout family. Thus FNO-LS does not enlarge the nonlinear approximation class; it only re-optimizes the linear coefficients supported by the learned representation.

A geometric interpretation is also useful.

\begin{proposition}\label{prop:projection}
Let
\[
\mathcal S:=\{X\theta:\theta\in\mathbb{R}^p\}\subset\mathbb{R}^{N_{\mathrm{row}}}
\]
be the column space of the frozen-feature design matrix. Then $X\widehat\theta_{\mathrm{LS}}$ is the orthogonal projection of $y$ onto $\mathcal S$.  Equivalently, the residual
\[
r:=y-X\widehat\theta_{\mathrm{LS}}
\]
satisfies
\[
X^\top r=0.
\]
If $X$ has full column rank, then
\[
\widehat\theta_{\mathrm{LS}}=(X^\top X)^{-1}X^\top y.
\]
\end{proposition}

\begin{proof}
This is the standard normal-equation characterization of least squares.
\end{proof}

It is essential to distinguish empirical optimality from generalization. Theorem~\ref{thm:empirical-monotonicity} concerns only the frozen-feature empirical squared loss on the training set. By itself, it yields no monotonicity statement for the test error or for the population risk on unseen obstacle realizations. On new samples, the refitted readout may improve accuracy, have little effect, or even perform worse, depending on the quality of the learned features, the amount of training data, and the stability with which the frozen-feature regression transfers beyond the training set.

The practical motivation for the Stage~2 readout refit is therefore more limited and more concrete. Stage~1 and Stage~2 solve different optimization problems: Stage~1 uses joint stochastic optimization of the backbone and the readout, whereas Stage~2 solves the induced full-data linear regression problem after the backbone has been frozen.

Because the frozen-feature least-squares system in Stage~2 can be very large, it is natural to solve it in a blockwise manner in order to reduce memory usage. We next summarize the additional linear-algebra cost of Stage~2, excluding the one-time forward sweep used to evaluate the frozen features on the training set.

\begin{remark}
Assume that the design matrix $X$ is processed in $K$ chunks
\[
X_m \in \mathbb{R}^{n_m \times p}, \qquad \sum_{m=1}^K n_m = N_{\mathrm{row}},
\]
and let $y_m \in \mathbb{R}^{n_m}$ denote the corresponding block of the target vector. Then Stage~2 can be implemented without explicitly storing the full matrix $X$.

\begin{enumerate}
    \item For the blockwise normal-equation approach, it suffices to accumulate
    \[
    G = \sum_{m=1}^K X_m^\top X_m \in \mathbb{R}^{p \times p},
    \qquad
    c = \sum_{m=1}^K X_m^\top y_m \in \mathbb{R}^{p}.
    \]
    This requires only $O(p^2)$ memory and
    \[
    O(N_{\mathrm{row}}p^2 + p^3)
    \]
    arithmetic, where the $O(p^3)$ term comes from the final dense solve.

    \item For a tall-skinny QR strategy, one computes local factorizations
    \[
    X_m = Q_m R_m,
    \]
    and forms a compressed least-squares problem from the stacked factors $R_m$ and vectors $Q_m^\top y_m$. This avoids forming $X^\top X$ explicitly and has the same leading-order arithmetic cost
    \[
    O(N_{\mathrm{row}}p^2)
    \]
    with respect to the number of rows.
\end{enumerate}
Hence, when $p$ is small, the additional cost of the readout refit is determined by a low-dimensional linear-algebra problem built from the frozen features, rather than by repeated backpropagation through the full FNO backbone. In particular, the full design matrix \(X\) need not be stored explicitly: one may process the training data in chunks and either accumulate the normal-equation summaries or maintain a compressed tall-skinny QR representation.
\end{remark}

From a numerical-linear-algebra viewpoint, QR is preferable when conditioning is a concern, whereas normal equations are more memory-efficient. In either case, once the frozen features have been generated, Stage~2 is a small deterministic readout-refitting problem in $p=d_r+1$ unknowns. It is most useful when the frozen backbone is already informative but the final affine readout is not yet fully optimized for those frozen features; its effect is naturally limited when end-to-end training has already brought the readout close to the frozen-feature empirical minimizer.

The readout refit also serves as a diagnostic for the trained FNO. A large improvement after refitting indicates that the frozen features are already informative but the final affine readout was not fully optimized. A marginal improvement indicates either that the end-to-end readout is already close to the frozen-feature least-squares optimum, or that the remaining error is mainly due to limitations of the learned nonlinear representation or the finite training set.

\section{Numerical experiments}\label{sec5}
\setcounter{equation}0
The experiments in this section are designed to assess not only overall field accuracy, but also whether the learned surrogates capture the geometric and inequality structure that is specific to obstacle problems. Throughout, we consider a fixed elliptic obstacle problem (\cite{Bierig2015}) in which only the obstacle field~\eqref{eq:obstacle_function} varies from sample to sample. Specifically, we take
\[
D=(-1,1)^2,
\qquad
a=1,
\qquad
f=-5,
\qquad
u_D=\frac12 \ \text{on }\partial D.
\]
For a given realization $g$ of the random obstacle field~\eqref{eq:obstacle_function}, define
\[
K_g:=\{v\in H^1(D)\,:\, v|_{\partial D}=u_D,\ \ v\ge g \ \text{a.e. in } D\}.
\]
The reference solution $u=\mathcal{G}(g)$ is the unique function in $K_g$ such that
\begin{equation}\label{Vi}
\int_D \nabla u\cdot \nabla(v-u)\,\mathrm{d}x
\ge
\int_D f(v-u)\,\mathrm{d}x
\qquad \forall v\in K_g.
\end{equation}

Although the abstract formulation in Section~\ref{sec2} is written for homogeneous Dirichlet data, the numerical experiments use the nonhomogeneous boundary condition \(u_D=1/2\) on \(\partial D\). Since \(u_D\) is constant, the shift \(w=u-u_D\) reduces this case to the homogeneous formulation with shifted obstacle \(g-u_D\).

For the finite-band self-affine random obstacle field~\eqref{eq:obstacle_function} introduced in Section~\ref{sec2}, the parameter \(\alpha\) serves as a global amplitude scale for the obstacle field. In the experiments below, we consider two benchmark regimes: a lower-amplitude case with \(\alpha=\pi/25\) and a higher-amplitude case with \(\alpha=\pi/5\). Since the modal envelope depends linearly on \(\alpha\), the lower-amplitude benchmark has an amplitude envelope that is smaller by a factor of \(5\), making the boundary compatibility condition substantially easier to satisfy. A brief numerical check of the generated samples also confirms that the obstacle values on the boundary remain below \(0.5\) in this regime.

In the higher-amplitude case \(\alpha=\pi/5\), a simple analytic estimate obtained by summing modal amplitudes is too conservative to yield a uniform realization-independent bound of the form \(g<1/2\) on \(\partial D\). For this reason, we do not claim such a sharp analytical bound in the higher-amplitude regime. Nevertheless, for the actual dataset used in the experiments, the largest observed obstacle value over all samples and grid points is \(0.289\), which is strictly smaller than \(0.5\). Hence every retained sample satisfies the discrete compatibility condition \(g(x_j)<0.5\) at all boundary grid points.

\begin{figure}[htbp]
    \centering
    \safeincludegraphics[width=0.8\linewidth]{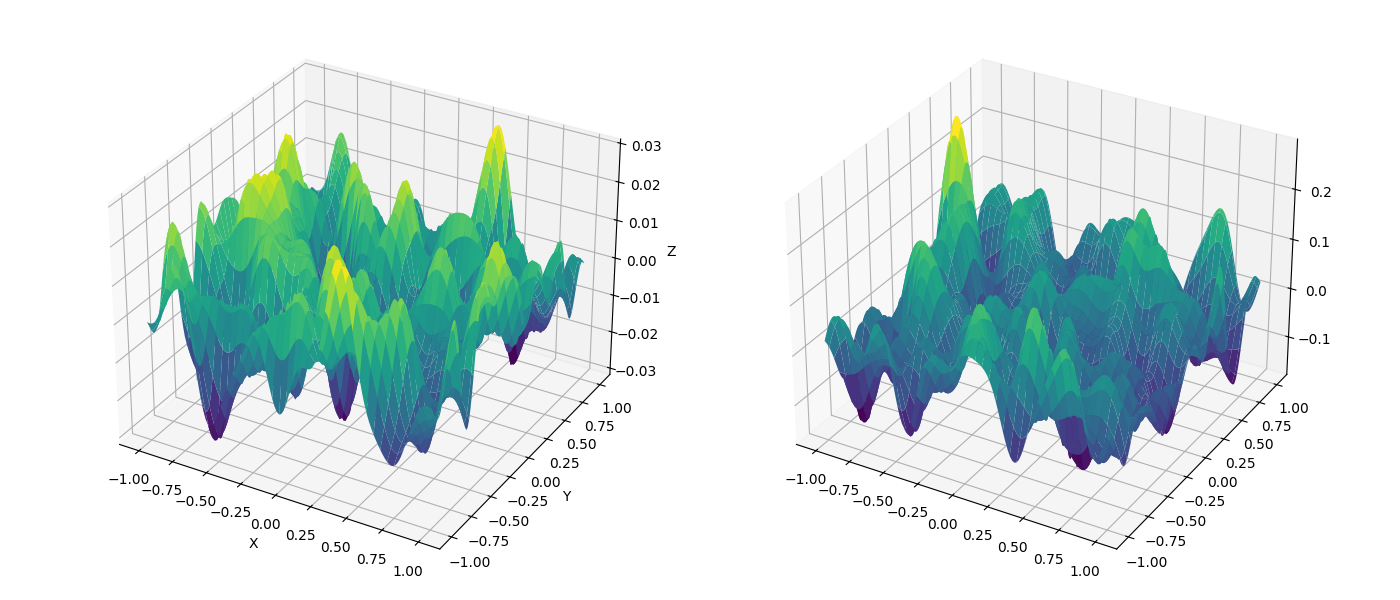}
    \caption{Representative realizations of the random obstacle field for two amplitude levels: $\alpha=\pi/25$ (left) and $\alpha=\pi/5$ (right).}
    \label{obstacle_compare}
\end{figure}

\subsection{Experimental setup}
The input is the random obstacle field $g$, and the output is the corresponding solution sampled on a uniform $128\times 128$ grid. Reference solutions are generated offline by a finite-difference active-set (\cite{Karkkainen2003,Hueber2005}) solver on a uniform \(128\times128\) grid over \(D=[-1,1]^2\), with mesh size \(h=2/(128-1)\). The active-set iteration is terminated when the active set is unchanged, or after \(1000\) iterations, and the neural operators are trained in PyTorch on the resulting input--output pairs. To improve reproducibility, all reported experiments are run with a fixed random seed. All computations are carried out on an NVIDIA RTX A6000 GPU.

We compare four model families: vanilla DeepONet, POD-DeepONet, two-stage DeepONet, and FNO. For vanilla DeepONet, the vectorized obstacle field on the $128\times 128$ grid is used as the branch input, while the spatial coordinate $x=(x_1,x_2)$ is used as the trunk input. In the reported vanilla DeepONet baseline, both the branch and trunk networks have four hidden layers of width $1368$; the branch input dimension is $16384$, corresponding to the vectorized $128\times128$ grid, and the trunk input dimension is $2$. This yields a parameter count that is approximately matched to that of the standard FNO configuration, so that the comparison is not driven mainly by model size. 

For POD-DeepONet, the trunk basis is replaced by a POD basis extracted only from the training solutions, and the branch network predicts only the modal coefficients (\cite{Lu2022Comparison}). For two-stage DeepONet, we follow the latent-code strategy of~\cite{Lee2024}. In the first stage, a coordinate-dependent decoder basis and sample-dependent latent coefficients are trained to reconstruct the reference solutions on the training set. In the second stage, a branch network is trained to map the input obstacle field to the learned latent coefficients. Both stages are trained for \(3000\) epochs. At test time, the predicted coefficients are combined with the learned decoder basis to produce the solution field.

For FNO-based models, the input consists of one obstacle channel together with two coordinate channels. Unless otherwise stated, the standard FNO configuration retains \(16\) Fourier modes in each spatial direction, uses channel width \(128\), padding size \(8\), and four Fourier layers. The resulting pointwise feature representation has dimension \(128\); after augmenting it with a bias term, the final affine readout has dimension \(p=129\). This is the readout used both in the standard FNO prediction head and in the least-squares refitting step of FNO-LS.

In all experiments, no input or output normalization is used. The batch size denotes the number of full input--output samples per optimization step, not the number of spatial grid points. We use batch size \(16\) for FNO-type models and batch size \(1024\) for the other baselines. All models are trained with AdamW with initial learning rate \(10^{-3}\), weight decay \(10^{-4}\). For FNO-type models, the learning rate is reduced by a factor of \(0.5\) every \(30\) epochs, while for the DeepONet-type baselines, it is reduced by a factor of \(0.5\) every \(200\) epochs. GELU activations are used throughout all neural network architectures.

In the tables below, rows marked ``LS'' report the FNO model after the Stage~2 post-training least-squares readout refit. The reported total time includes both Stage~1 end-to-end FNO training and Stage~2 readout refitting, while the additional Stage~2 cost is reported separately in the text. The parameter count shown in the tables is the total parameter count of the underlying FNO architecture; only the final 129-parameter affine readout is refit in Stage~2.

Let $\{(g^{(i)},u^{(i)})\}_{i=1}^{N_{\mathrm{test}}}$ be the test set, and let $\widehat u^{(i)}$ denote the prediction for the $i$th sample. We evaluate the models by three complementary criteria: a field error, a contact-set metric, and an obstacle-violation metric. This choice reflects the mathematical structure of obstacle problems. A good surrogate should not only approximate the solution field in bulk, but should also identify the contact region reliably and avoid violating the inequality constraint.

All \(L^2(D)\) norms reported below are evaluated as discrete quadrature norms on the uniform \(128\times128\) grid. First, we measure the relative \(L^2(D)\)
error by
\begin{equation}
    e_{L^2}^{(i)}
:=
\frac{\|u^{(i)}-\widehat u^{(i)}\|_{L^2(D)}}{\|u^{(i)}\|_{L^2(D)}}.
\end{equation}
This is the primary field error metric and quantifies the overall accuracy of the predicted solution over the whole domain.

Second, since a central structural feature of the obstacle problem is the contact set
\[
\mathcal C^{(i)}:=\{x\in D:\ u^{(i)}(x)=g^{(i)}(x)\},
\]
we also evaluate how well the predicted contact region is recovered. On the discrete grid, exact equality is not meaningful numerically, so we introduce a tolerance-based contact-set surrogate. For each test sample, we define
\begin{equation}
    \tau_i
:=
\max\Bigl\{\tau_{\mathrm{abs}},\,
\eta \max\bigl(
\|u^{(i)}\|_{L^\infty(D)},\,
\|g^{(i)}\|_{L^\infty(D)}
\bigr)\Bigr\},
\end{equation}
where $\eta>0$ is a prescribed relative tolerance and $\tau_{\mathrm{abs}}>0$ is a small absolute tolerance. In this paper, we use $\eta = 10^{-2}$ and $\tau_{\mathrm{abs}} = 10^{-5}$. Using $\tau_i$, we define the discrete tolerance-based contact sets
\begin{equation}
    \mathcal C_{\tau_i}^{\mathrm{ref}}
:=
\{x\in D:\ |u^{(i)}(x)-g^{(i)}(x)|\le \tau_i\},
\qquad
\mathcal C_{\tau_i}^{\mathrm{pred}}
:=
\{x\in D:\ |\widehat u^{(i)}(x)-g^{(i)}(x)|\le \tau_i\}.
\end{equation}

We then measure their overlap by the intersection-over-union (IoU), with the convention
\[
\mathrm{IoU}^{(i)}
:=
\begin{cases}
1, 
& \text{if }  \mathcal C_{\tau_i}^{\mathrm{ref}}\cup \mathcal C_{\tau_i}^{\mathrm{pred}} = \varnothing, \\[0.8em]
\dfrac{
|\mathcal C_{\tau_i}^{\mathrm{ref}}\cap \mathcal C_{\tau_i}^{\mathrm{pred}}|
}{
|\mathcal C_{\tau_i}^{\mathrm{ref}}\cup \mathcal C_{\tau_i}^{\mathrm{pred}}|
},
& \text{otherwise}.
\end{cases}
\]

This metric evaluates geometric agreement of the predicted and reference contact regions. It adds information beyond the field error: because the contact-set boundary is a free boundary, a visibly misplaced contact region can coexist with a relatively small global $L^2$ error.

Third, because the obstacle problem imposes the unilateral constraint $u^{(i)}\ge g^{(i)}$, we also measure how strongly the prediction violates the obstacle. Define the pointwise violation by
\[
v^{(i)}(x):=(g^{(i)}(x)-\widehat u^{(i)}(x))_+,
\qquad (a)_+:=\max\{a,0\},
\]
and the relative obstacle-violation error by
\begin{equation}
e_{\mathrm{viol}}^{(i)}
:=
\frac{\| (g^{(i)}-\widehat u^{(i)})_+ \|_{L^2(D)}}{\|u^{(i)}\|_{L^2(D)}}.
\end{equation}
This quantity is zero if the prediction stays above the obstacle everywhere, and it becomes larger when the predicted solution penetrates the obstacle more severely. It is therefore complementary to the $L^2$ error and directly probes whether the learned surrogate respects the defining inequality structure of the problem.

For the test set, we report the average and worst-case relative $L^2$ errors,
\[
E_{L^2}^{\mathrm{avg}}
:=
\frac{1}{N_{\mathrm{test}}}\sum_{i=1}^{N_{\mathrm{test}}} e_{L^2}^{(i)},
\qquad
E_{L^2}^{\max}
:=
\max_{1\le i\le N_{\mathrm{test}}} e_{L^2}^{(i)},
\]
and, when contact-set statistics are reported, we also record the average IoU and the average relative obstacle-violation error over the test set,
\[
E_{\mathrm{IoU}}^{\mathrm{avg}}
:=
\frac{1}{N_{\mathrm{test}}}\sum_{i=1}^{N_{\mathrm{test}}}\mathrm{IoU}^{(i)},
\qquad
E_{\mathrm{viol}}^{\mathrm{avg}}
:=
\frac{1}{N_{\mathrm{test}}}\sum_{i=1}^{N_{\mathrm{test}}} e_{\mathrm{viol}}^{(i)}.
\]

\subsection{Example 1: lower-amplitude obstacles}

We first consider the lower-amplitude benchmark with $\alpha=\pi/25$. The training set contains $10{,}000$ obstacle realizations together with their reference solutions, and the test set contains $N_{\mathrm{test}}=8{,}000$ additional samples. All fields are sampled on a $128\times 128$ uniform grid.

\begin{table}[htbp]
\centering
\caption{Comparison of neural-operator models on the lower-amplitude benchmark (\(\alpha=\pi/25\)). The parameter counts reported in the table are total model parameter counts. For rows marked ``LS'', Stage~2 refits only the final 129-parameter affine readout, and the reported time is the total time of Stage~1 plus Stage~2.}
\label{table_example1}
\scriptsize  
\begin{tabular}{lcccccc}
\toprule
Method & Total params. & Total time (s) & $E_{L^2}^{\mathrm{avg}} $ &  $E_{L^2}^{\max}$   & $E_{\mathrm{IoU}}^{\mathrm{avg}}$ & $E_{\mathrm{viol}}^{\mathrm{avg}}$ \\
\midrule
Vanilla DeepONet         & 33,655,537 & 2,726.43  & 4.95E-03 & 2.54E-02 & 0.910454 & 3.32E-03 \\
POD-DeepONet (128 modes) & 30,081,080 & \textbf{1,128.29}  & 7.94E-04 & 9.51E-03 & 0.994581 & 3.34E-04 \\
POD-DeepONet (256 modes) & 30,256,312 & 1,135.98  & 7.99E-04 & 9.52E-03 & 0.994546 & 3.39E-04 \\
Two-stage DeepONet      & 47,346,905 & 4,573.01  & 4.17E-03 & 2.41E-02 & 0.932494 & 2.72E-03 \\
FNO-100                 & 33,637,633 & 7,882.84  & 7.70E-04 & \textbf{4.23E-03} & 0.996686 & 2.67E-04 \\
FNO-100-LS (NE)       & 33,637,633 & 7,924.98  & \textbf{6.37E-04} & 4.24E-03 & \textbf{0.996824} & \textbf{2.44E-04} \\
FNO-100-LS (QR)       & 33,637,633 & 7,977.99  & \textbf{6.37E-04} & 4.24E-03 & \textbf{0.996824} & \textbf{2.44E-04} \\
\bottomrule
\end{tabular}
\end{table}

Here and below, ``FNO-\(k\)'' denotes the same FNO architecture trained for \(k\) epochs; only the training duration differs across these rows. Table~\ref{table_example1} summarizes the model ranking on this benchmark. For the ``-LS'' rows, the reported time includes both Stage~1 end-to-end FNO training and Stage~2 readout refitting. The additional Stage~2 cost is \(42.14\)~s for the normal-equation implementation and \(95.15\)~s for the QR implementation. Among the methods tested on this lower-amplitude benchmark, the FNO variants perform best overall. In particular, FNO-100 gives the smallest value of \(E_{L^2}^{\max}\),  while FNO-100-LS attains the best average relative \(L^2\) error, average contact-set IoU, and average obstacle-violation error.

At this lower amplitude level, the solution family is still sufficiently compressible for a POD-based representation to remain effective. The obstacle is nontrivial and the contact geometry is already moving from sample to sample, but the induced free-boundary complexity is still limited enough that a reduced basis extracted from training solutions remains well aligned with the dominant solution manifold. Overall, FNO and POD-DeepONet are both highly accurate in this regime; FNO-LS achieves the best average field error and structural metrics, while POD-DeepONet remains faster in the reported implementation.

\begin{table}[htbp]
\centering
\caption{Effect of the post-training least-squares readout refit for FNO at different training stages on the lower-amplitude benchmark \((\alpha=\pi/25)\). The parameter counts reported in the table are total model parameter counts. For rows marked ``LS'', only the final 129-parameter affine readout is refit in Stage~2.}
\label{table_example1_ls}
\scriptsize  
\begin{tabular}{lcccccc}
\toprule
Method & Total params. & Total time (s) & $E_{L^2}^{\mathrm{avg}} $ &  $E_{L^2}^{\max}$   & $E_{\mathrm{IoU}}^{\mathrm{avg}}$ & $E_{\mathrm{viol}}^{\mathrm{avg}}$ \\
\midrule
FNO-200                 & 33,637,633 & 15,773.42 & 5.67E-04 & 3.91E-03 & \textbf{0.997231} & 2.21E-04 \\
FNO-200-LS (NE)       & 33,637,633 & 15,815.53 & \textbf{5.47E-04} & \textbf{3.90E-03} & 0.997157 & \textbf{2.14E-04} \\
FNO-200-LS (QR)       & 33,637,633 & 15,868.74 & \textbf{5.47E-04} & \textbf{3.90E-03} & 0.997157 & \textbf{2.14E-04} \\
FNO-100                 & 33,637,633 & 7,882.84  & 7.70E-04 & 4.23E-03 & 0.996686 & 2.67E-04 \\
FNO-100-LS (NE)       & 33,637,633 & 7,924.98  & 6.37E-04 & 4.24E-03 & 0.996824 & 2.44E-04 \\
FNO-100-LS (QR)       & 33,637,633 & 7,977.99  & 6.37E-04 & 4.24E-03 & 0.996824 & 2.44E-04 \\
FNO-50                  & 33,637,633 & 3,946.28  & 1.55E-03 & 6.75E-03 & 0.993331 & 5.91E-04 \\
FNO-50-LS (NE)        & 33,637,633 & 3,984.13  & 1.13E-03 & 6.77E-03 & 0.994936 & 3.85E-04 \\
FNO-50-LS (QR)        & 33,637,633 & 4,037.44  & 1.13E-03 & 6.77E-03 & 0.994936 & 3.85E-04 \\
FNO-20                  & 33,637,633 & 1,581.09  & 2.59E-03 & 1.14E-02 & 0.990578 & 8.46E-04 \\
FNO-20-LS (NE)        & 33,637,633 & 1,623.26  & 2.08E-03 & 1.07E-02 & 0.991869 & 5.37E-04 \\
FNO-20-LS (QR)        & 33,637,633 & 1,676.49  & 2.08E-03 & 1.07E-02 & 0.991869 & 5.37E-04 \\
FNO-10                  & 33,637,633 & \textbf{793.31}    & 5.27E-03 & 1.39E-02 & 0.981050 & 1.02E-03 \\
FNO-10-LS (NE)        & 33,637,633 & 835.44    & 2.39E-03 & 1.37E-02 & 0.989119 & 6.50E-04 \\
FNO-10-LS (QR)        & 33,637,633 & 888.63    & 2.39E-03 & 1.37E-02 & 0.989119 & 6.50E-04 \\
\bottomrule
\end{tabular}
\end{table}

Table~\ref{table_example1_ls} addresses a different question, namely the role of the post-training least-squares readout refit for FNO at different training stages. A consistent pattern emerges on this benchmark. The readout refit provides its largest gains for partially trained FNO backbones: the average error is reduced substantially for FNO-10, FNO-20, and FNO-50 after refit, and the same trend is visible in the contact-set IoU and obstacle-violation metrics. By contrast, once end-to-end training is already close to convergence, the additional improvement becomes small. This matches the interpretation developed in Section~3. The second stage does not change the learned nonlinear representation; it only recomputes the least-squares final affine readout for the frozen features.
These results suggest that the LS refit mainly reduces readout optimization error rather than representation error. Its effect is largest in the early or intermediate training regime, where the FNO backbone has already learned useful features but the final affine readout is not yet fully aligned with them. Near convergence, the readout obtained by end-to-end training is already close to the frozen-feature least-squares solution, and the remaining error is primarily due to the learned representation itself or to the finite training set.

The gains are also more consistent for average error than for worst-case error. This is natural, since the least-squares readout refit minimizes a global squared-error objective over all training rows and is therefore designed to improve the average fit of the final affine readout, not to target the single worst test sample directly. Accordingly, the data support average-case improvement, especially for partially trained backbones, but they do not suggest that the readout refit resolves worst-case behavior in a systematic sense. The NE and QR implementations yield identical predictive metrics up to the displayed precision, as expected since they solve the same frozen-feature regression problem by different numerical linear-algebra routes.

\begin{figure}[htbp]
    \centering
    \safeincludegraphics[width=0.8\linewidth]{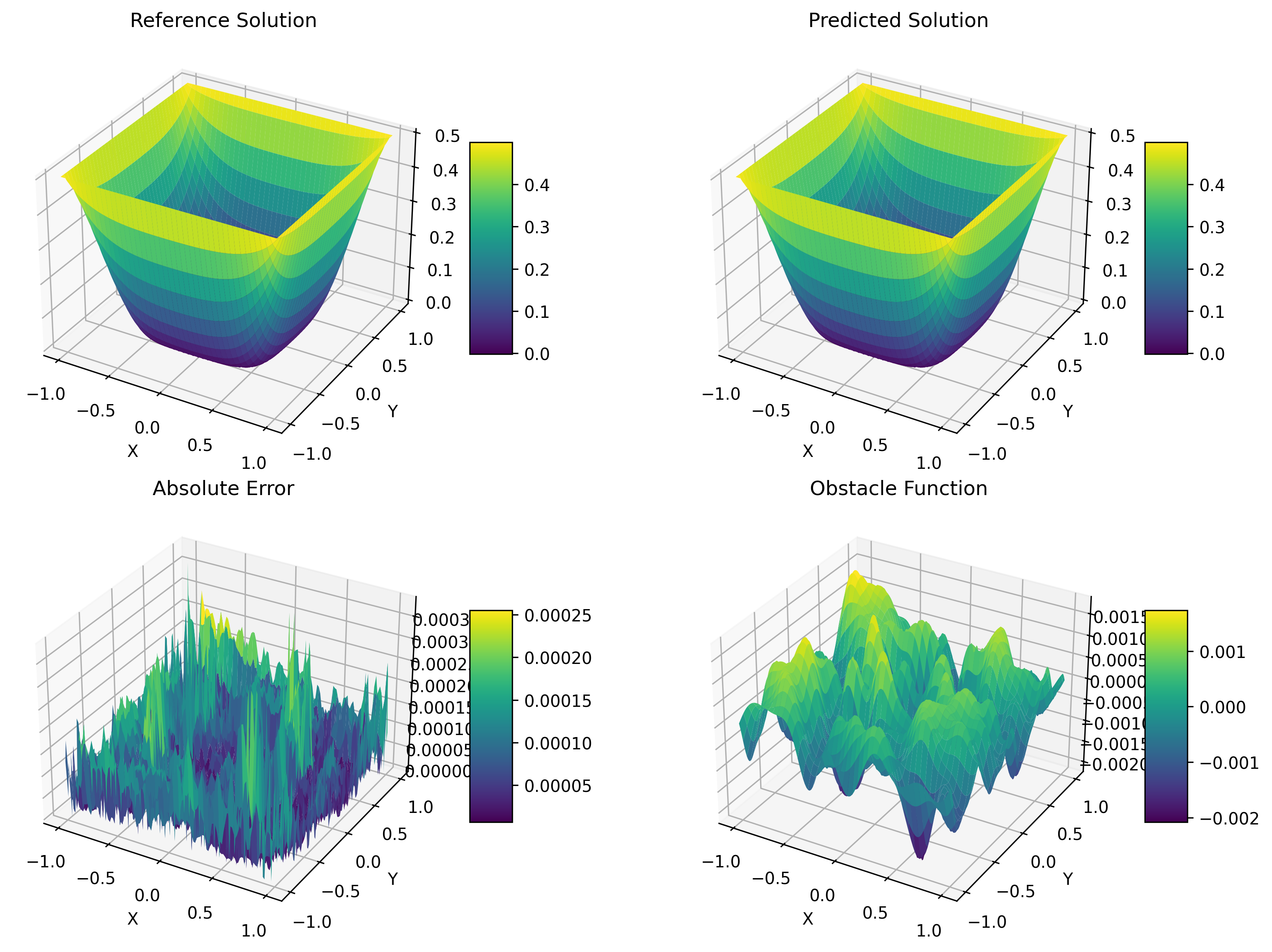}
    \caption{Best-case FNO-100 prediction on the test set in Example~1, selected by the smallest sample-wise relative $L^2$ error. The four panels show the reference solution, the predicted solution, the absolute error, and the obstacle field. For this sample, the relative $L^2$ error is $0.000380$ and the relative obstacle-violation error is $0.000135$.}
    \label{exm1_best}
\end{figure}

\begin{figure}[htbp]
    \centering
    \safeincludegraphics[width=0.8\linewidth]{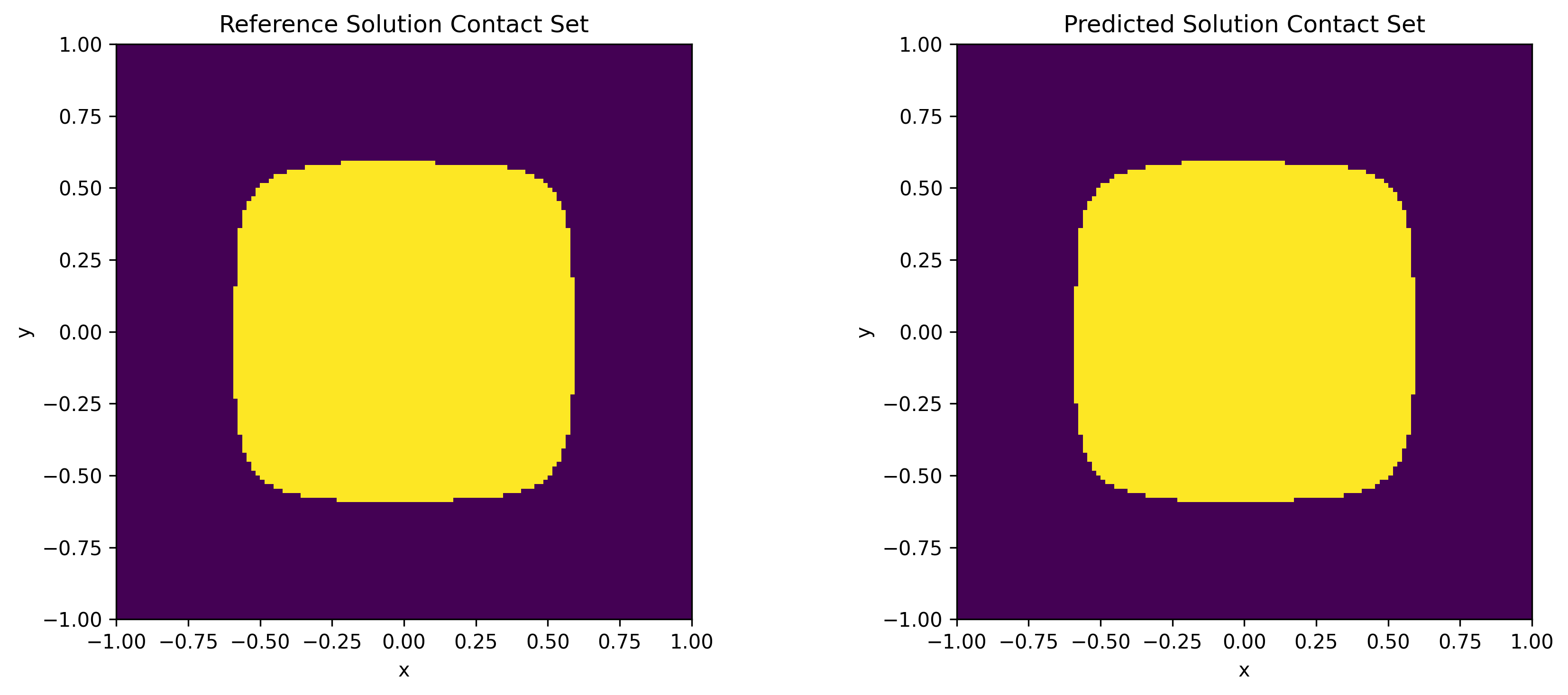}
    \caption{Comparison of the reference and predicted contact sets for the  best-case sample shown in Figure~\ref{exm1_best}. The reference contact set is shown on the left and the predicted contact set on the right; contact regions are highlighted in yellow. The contact-set IoU is $0.998497$.}
    \label{exm1_best_contact}
\end{figure}

\begin{figure}[htbp]
    \centering
    \safeincludegraphics[width=0.8\linewidth]{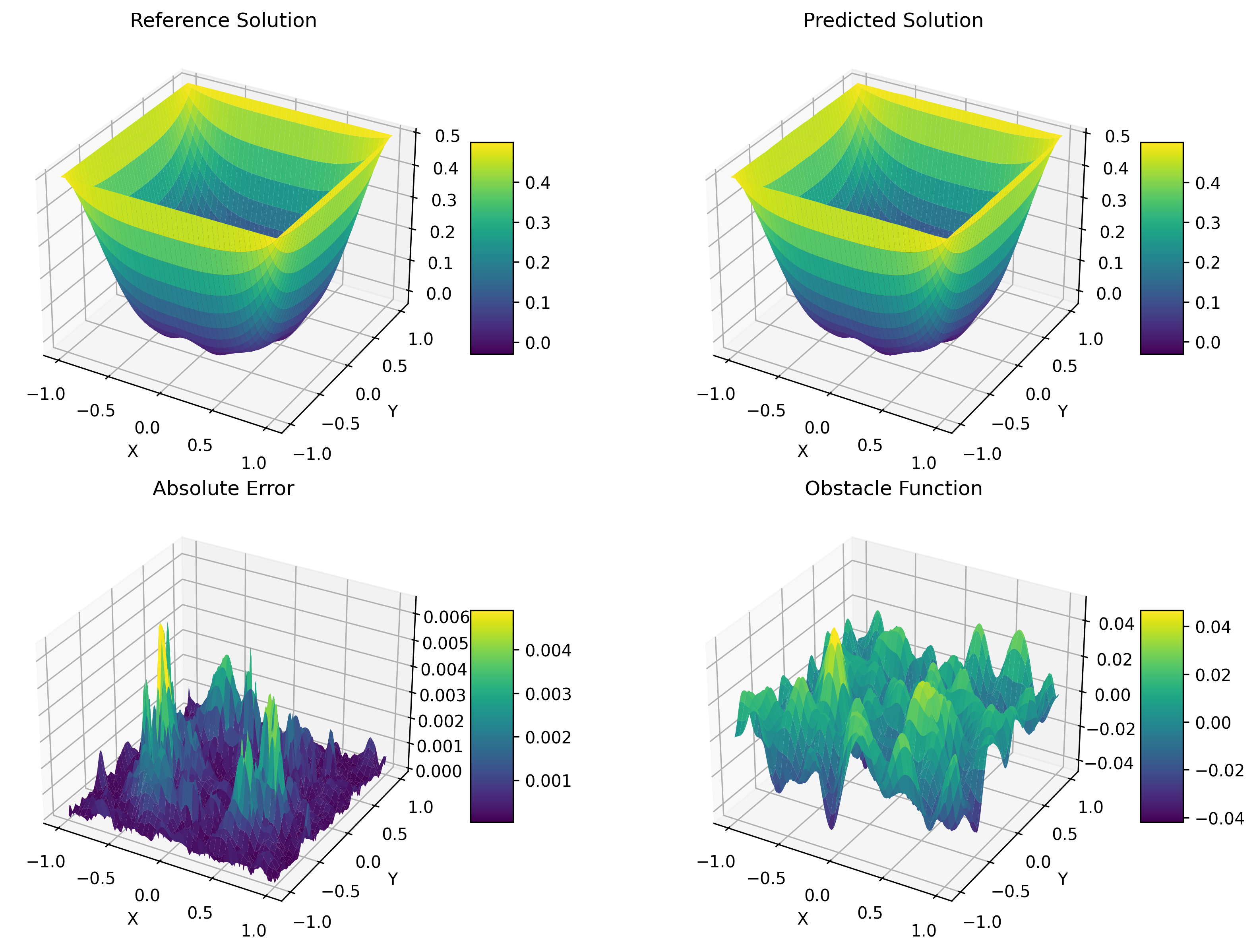}
    \caption{Worst-case FNO-100 prediction on the test set in Example 1, selected by the largest sample-wise relative $L^2$ error. The four panels show the reference solution, the predicted solution, the absolute error, and the obstacle field. For this sample, the relative $L^2$ error is $0.004233$ and the relative obstacle-violation error is $0.001125$.}
    \label{exm1_worst}
\end{figure}

\begin{figure}[htbp]
    \centering
    \safeincludegraphics[width=0.8\linewidth]{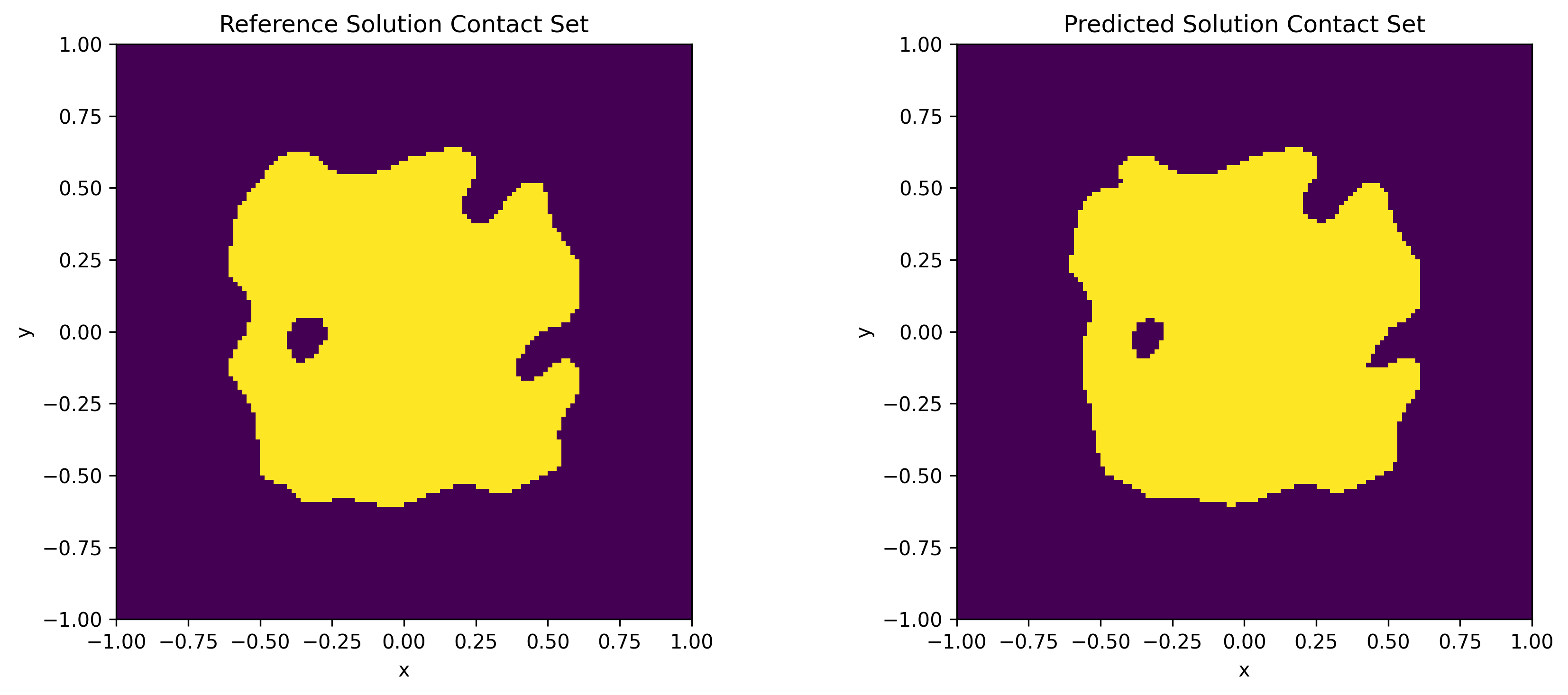}
    \caption{Comparison of the reference and predicted contact sets for the  worst-case sample shown in Figure~\ref{exm1_worst}. The reference contact set is shown on the left and the predicted contact set on the right; contact regions are highlighted in yellow. The contact-set IoU is $0.958628$.}
    \label{exm1_contact}
\end{figure}

Figures~\ref{exm1_worst}--\ref{exm1_contact} complement the aggregate metrics by showing where the remaining errors occur. Even on the worst-case sample in this regime, the predicted solution captures the global shape of the reference solution well. The dominant error is localized near regions where the contact set changes geometry, indicating a local free-boundary misalignment rather than a large-scale distortion of the solution field.

\subsection{Example 2: higher-amplitude obstacles}

We next increase the obstacle amplitude to \(\alpha=\pi/5\), which yields a higher-amplitude benchmark with more pronounced oscillatory features and a more challenging input field. The training and test sets have the same sizes and resolution as in Example 1.

\begin{table}[htbp]
\centering
\caption{Comparison of neural-operator models on the higher-amplitude benchmark \((\alpha=\pi/5)\). The parameter counts reported in the table are total model parameter counts. For rows marked ``LS'', Stage~2 refits only the final 129-parameter affine readout, and the reported time is the total time of Stage~1 plus Stage~2.}
\label{table_example2}
\scriptsize  
\begin{tabular}{lcccccc}
\toprule
Method & Total params. & Total time (s) & $E_{L^2}^{\mathrm{avg}} $ &  $E_{L^2}^{\max}$   & $E_{\mathrm{IoU}}^{\mathrm{avg}}$ & $E_{\mathrm{viol}}^{\mathrm{avg}}$ \\
\midrule
Vanilla DeepONet             & 33,655,537 & 2,573.96  & 1.32E-02 & 1.01E-01 & 0.773510 & 7.18E-03 \\
POD-DeepONet (256 modes)     & 30,256,312 & \textbf{1,070.18}  & 9.48E-03 & 8.55E-02 & 0.840091 & 4.23E-03 \\
POD-DeepONet (512 modes)     & 30,606,776 & 1,097.56  & 9.48E-03 & 8.27E-02 & 0.840934 & 4.23E-03 \\
Two-stage DeepONet           & 47,346,905 & 5,899.65  & 1.28E-02 & 1.06E-01 & 0.777642 & 6.49E-03 \\
FNO-100                      & 33,637,633 & 7,904.27  & 1.22E-03 & 1.04E-02 & 0.990162 & 3.95E-04 \\
FNO-100-LS (NE)            & 33,637,633  &7,946.37  & \textbf{1.08E-03} & \textbf{1.01E-02} & \textbf{0.990738} & \textbf{3.46E-04} \\
FNO-100-LS (QR)            & 33,637,633  & 7,999.52 & \textbf{1.08E-03} & \textbf{1.01E-02} & \textbf{0.990738} & \textbf{3.46E-04} \\
\bottomrule
\end{tabular}
\end{table}

Table~\ref{table_example2} reports the higher-amplitude benchmark, which is the more demanding test.  In this benchmark, the performance gap widens. FNO substantially outperforms all tested DeepONet-type baselines in average relative $L^2$ error, worst-case relative $L^2$ error, average contact-set IoU, and average obstacle-violation error. The same qualitative ranking is visible across the three metrics: the FNO predictions not only reduce bulk field error, but also recover the contact geometry more accurately and exhibit less penetration of the obstacle. For the ``-LS'' rows, the reported time includes both Stage~1 end-to-end FNO training and Stage~2 readout refitting. The additional Stage~2 cost is \(42.10\)~s for the normal-equation implementation and \(95.25\)~s for the QR implementation.

The change from the lower- to the higher-amplitude benchmark is consistent with the structure of obstacle problems. As the obstacle amplitude increases within this self-affine family, the oscillatory features become more pronounced and induce more mobile and more intricate contact regions. In this regime, the solution family is less well captured by a fixed low-dimensional linear representation. In particular, moving contact regions, thin branches, small isolated components, and highly curved segments of the free boundary are difficult to represent accurately by a basis that is fixed once and for all. The higher-amplitude results therefore suggest that the adaptive grid-based features learned by FNO are better matched to this benchmark than the fixed-basis or branch--trunk representations tested here.

The three metrics are especially important in this regime. A gap in relative \(L^2\) error alone would already indicate a substantial difference in field accuracy, but the contact-set IoU and obstacle-violation metrics show that the observed advantage is not merely one of bulk regression. The larger obstacle amplitude makes geometric recovery of the contact set and
feasibility with respect to the unilateral constraint significantly harder. Thus, on the present higher-amplitude benchmark, the FNO advantage extends beyond field accuracy to structural properties that are specific to obstacle problems. This should still be read as a benchmark-specific conclusion rather than as a universal claim about all obstacle problems, but it is supported by the present data.

\begin{table}[htbp]
\centering
\caption{Effect of the post-training least-squares readout refit for FNO at different training stages on the higher-amplitude benchmark \((\alpha=\pi/5)\). The parameter counts reported in the table are total model parameter counts. For rows marked ``LS'', only the final 129-parameter affine readout is refit in Stage~2.}
\label{table_example2_ls}
\scriptsize 
\begin{tabular}{lcccccc}
\toprule
Method & Total params. & Total time (s) & $E_{L^2}^{\mathrm{avg}} $ &  $E_{L^2}^{\max}$   & $E_{\mathrm{IoU}}^{\mathrm{avg}}$ & $E_{\mathrm{viol}}^{\mathrm{avg}}$ \\
\midrule
FNO-200                    & 33,637,633 & 15,790.50 & 9.26E-04 & 9.18E-03 & 0.991628 & 3.04E-04 \\
FNO-200-LS (NE)            & 33,637,633 & 15,832.64 & \textbf{9.06E-04} & \textbf{9.14E-03} & \textbf{0.991705} & \textbf{2.89E-04} \\
FNO-200-LS (QR)            & 33,637,633 & 15,885.76 & \textbf{9.06E-04} & \textbf{9.14E-03} & \textbf{0.991705} & \textbf{2.89E-04} \\
FNO-100                    & 33,637,633 & 7,904.27  & 1.22E-03 & 1.04E-02 & 0.990162 & 3.95E-04 \\
FNO-100-LS (NE)            & 33,637,633 & 7,946.37  & 1.08E-03 & 1.01E-02 & 0.990738 & 3.46E-04 \\
FNO-100-LS (QR)            & 33,637,633 & 7,999.52  & 1.08E-03 & 1.01E-02 & 0.990738 & 3.46E-04 \\
FNO-50                     & 33,637,633 & 3,961.86  & 2.59E-03 & 1.29E-02 & 0.986638 & 7.55E-04 \\
FNO-50-LS (NE)             & 33,637,633 & 4,003.95  & 1.62E-03 & 1.20E-02 & 0.986790 & 4.98E-04 \\
FNO-50-LS (QR)             & 33,637,633 & 4,057.07  & 1.62E-03 & 1.20E-02 & 0.986790 & 4.98E-04 \\
FNO-20                     & 33,637,633 & 1,589.23  & 4.67E-03 & 1.86E-02 & 0.972967 & 1.02E-03 \\
FNO-20-LS (NE)             & 33,637,633 & 1,631.39  & 2.78E-03 & 1.61E-02 & 0.977798 & 8.57E-04 \\
FNO-20-LS (QR)             & 33,637,633 & 1,684.43  & 2.78E-03 & 1.61E-02 & 0.977798 & 8.57E-04 \\
FNO-10                     & 33,637,633 & \textbf{797.14}    & 6.66E-03 & 2.30E-02 & 0.959126 & 2.84E-03 \\
FNO-10-LS (NE)             & 33,637,633 & 839.20    & 3.77E-03 & 2.09E-02 & 0.962871 & 1.21E-03 \\
FNO-10-LS (QR)             & 33,637,633 & 892.38    & 3.77E-03 & 2.09E-02 & 0.962871 & 1.21E-03 \\
\bottomrule
\end{tabular}
\end{table}

Table~\ref{table_example2_ls} shows that the qualitative role of the least-squares readout refit persists in the higher-amplitude regime. Again, the readout refit is most useful for partially trained FNO backbones, whereas the gains become modest once the end-to-end model is already well trained. For example, the average error drops substantially for FNO-10 and FNO-20 after refitting, whereas the improvement from FNO-200 to FNO-200-LS is comparatively small. 

This again supports the diagnostic interpretation of the readout refit. In the early and intermediate training regimes, the FNO backbone has already extracted useful grid-based features, but the final affine readout can still be improved by a full-data least-squares solve. Once the backbone and readout have been jointly optimized for sufficiently many epochs, the LS correction becomes smaller, indicating that the remaining error is less attributable to readout suboptimality.

The POD-DeepONet baselines in Table~\ref{table_example2} use 256 and 512 POD modes, whereas the standard FNO uses 128 learned latent features in its final pointwise readout. These are not identical objects: the former are fixed linear output bases, while the latter are data-adaptive features learned end to end, so the comparison should not be overstated. Nevertheless, the comparison is still  meaningful at the level of effective representation efficiency. From this viewpoint, the results suggest that, on the present higher-amplitude benchmark, the learned FNO representation is more effective than the fixed POD bases considered here.

\begin{figure}[htbp]
    \centering
    \safeincludegraphics[width=0.8\linewidth]{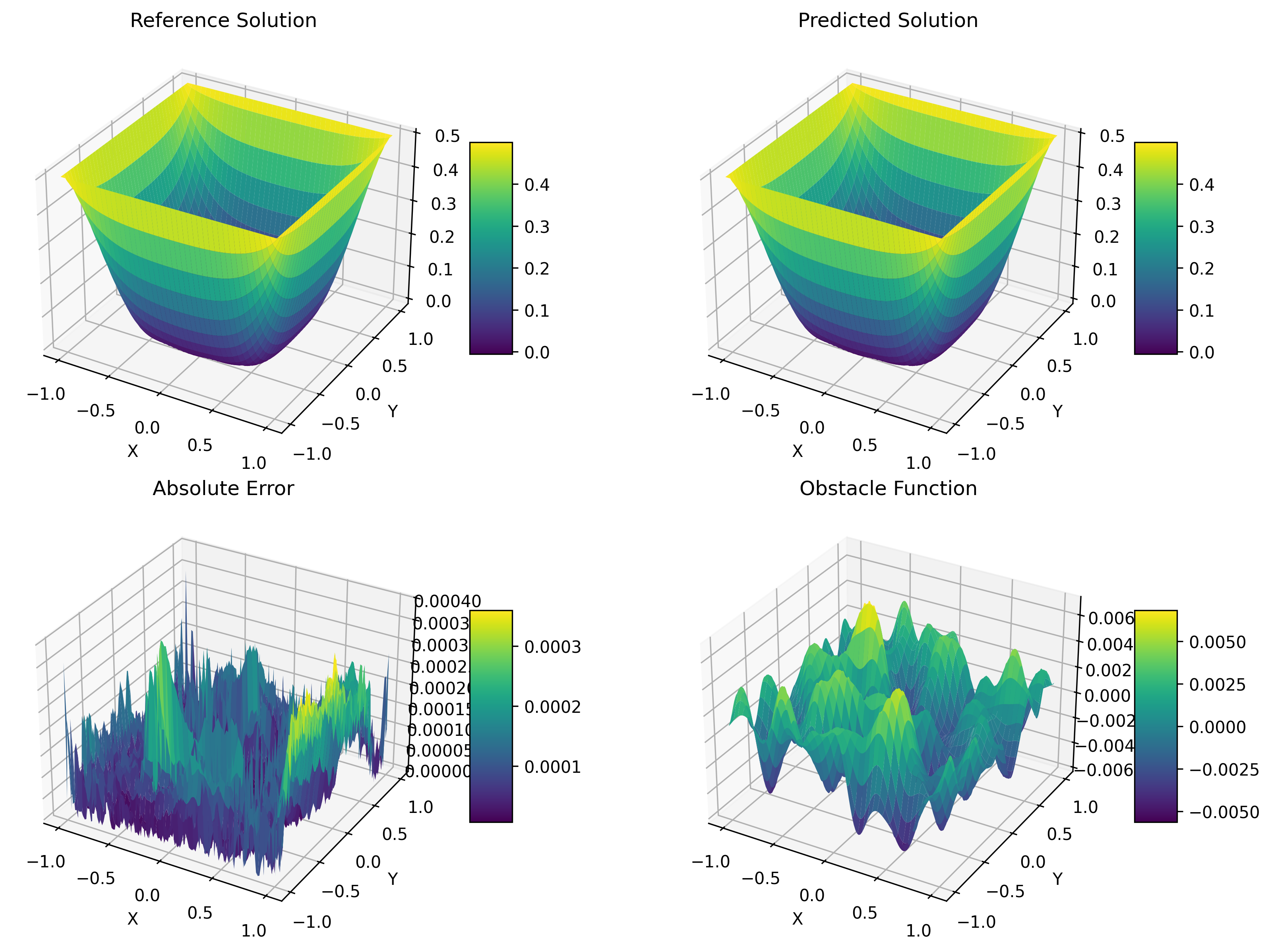}
    \caption{Best-case FNO-100 prediction on the test set in Example~2, selected by the smallest sample-wise relative $L^2$ error. The four panels show the reference solution, the predicted solution, the absolute error, and the obstacle field. For this sample, the relative $L^2$ error is $0.000453$ and the relative obstacle-violation error is $0.000188$.}
    \label{exm2_best}
\end{figure}

\begin{figure}[htbp]
    \centering
    \safeincludegraphics[width=0.8\linewidth]{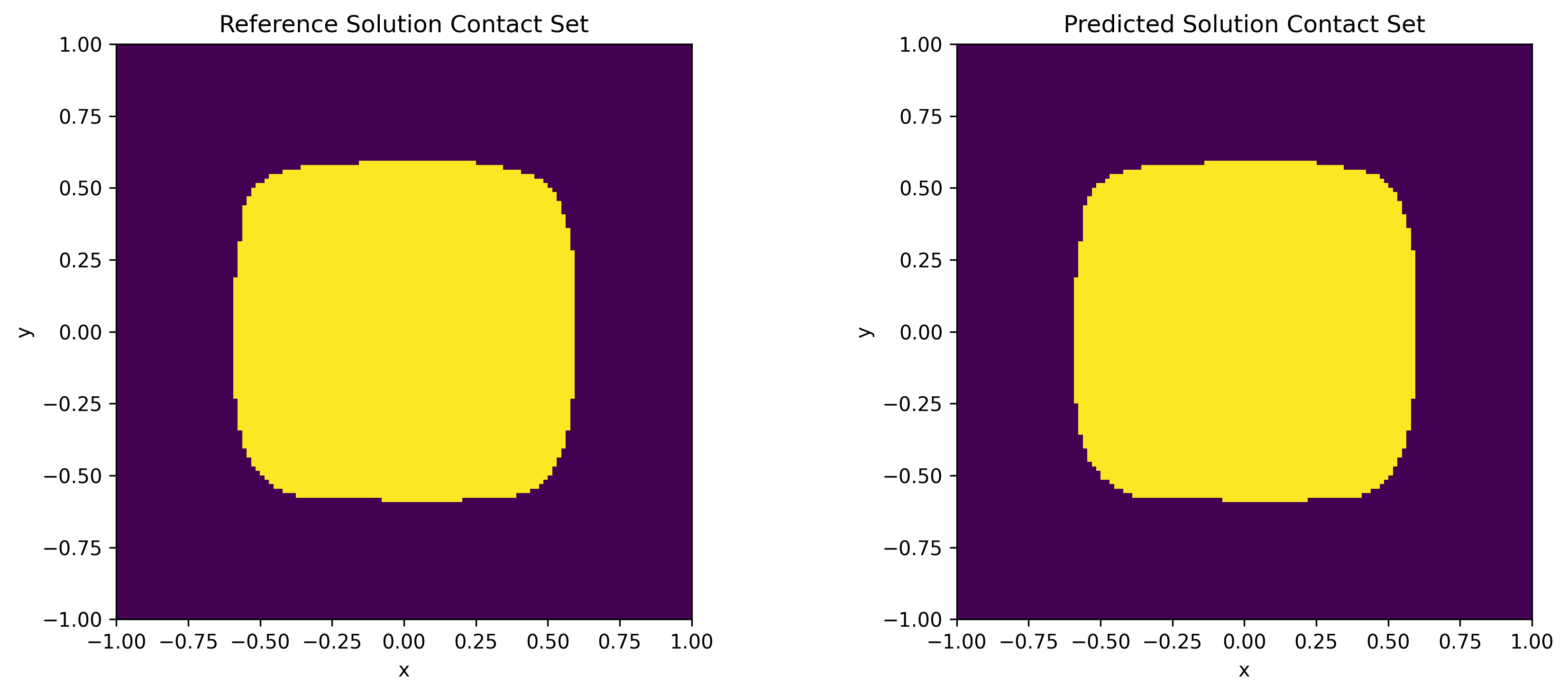}
    \caption{Comparison of the reference and predicted contact sets for the best-case sample shown in Figure~\ref{exm2_best}. The reference contact set is shown on the left and the predicted contact set on the right; contact regions are highlighted in yellow. The contact-set IoU is $0.998312$.}
    \label{exm2_best_contact}
\end{figure}

\begin{figure}[htbp]
    \centering
    \safeincludegraphics[width=0.8\linewidth]{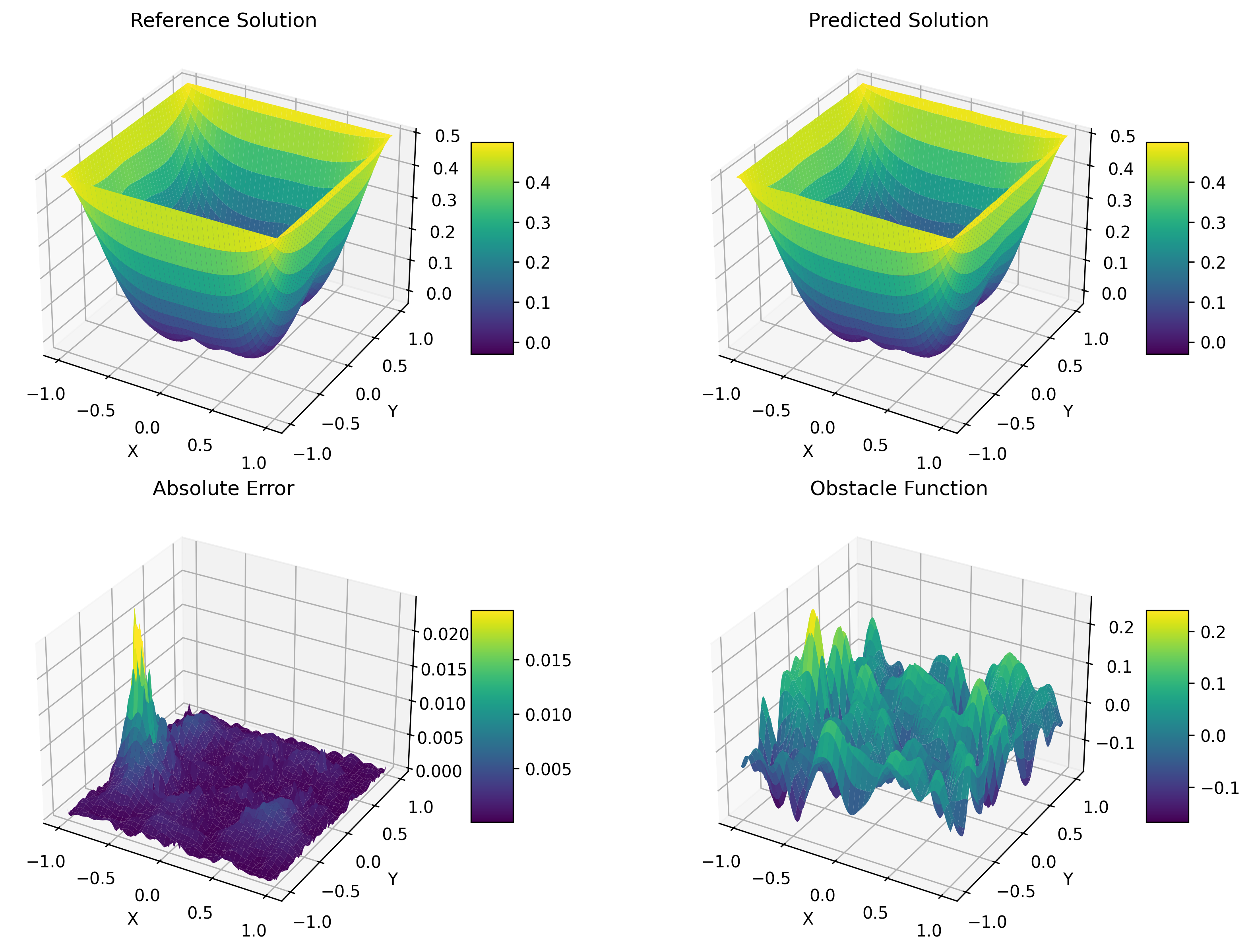}
    \caption{Worst-case FNO-100 prediction on the test set in Example~2, selected by the largest sample-wise relative $L^2$ error. The four panels show the reference solution, the predicted solution, the absolute error, and the obstacle field. For this sample, the relative $L^2$ error is $0.010426$ and the relative obstacle-violation error is $0.003104$.}
    \label{exm2_worst}
\end{figure}

\begin{figure}[htbp]
    \centering
    \safeincludegraphics[width=0.8\linewidth]{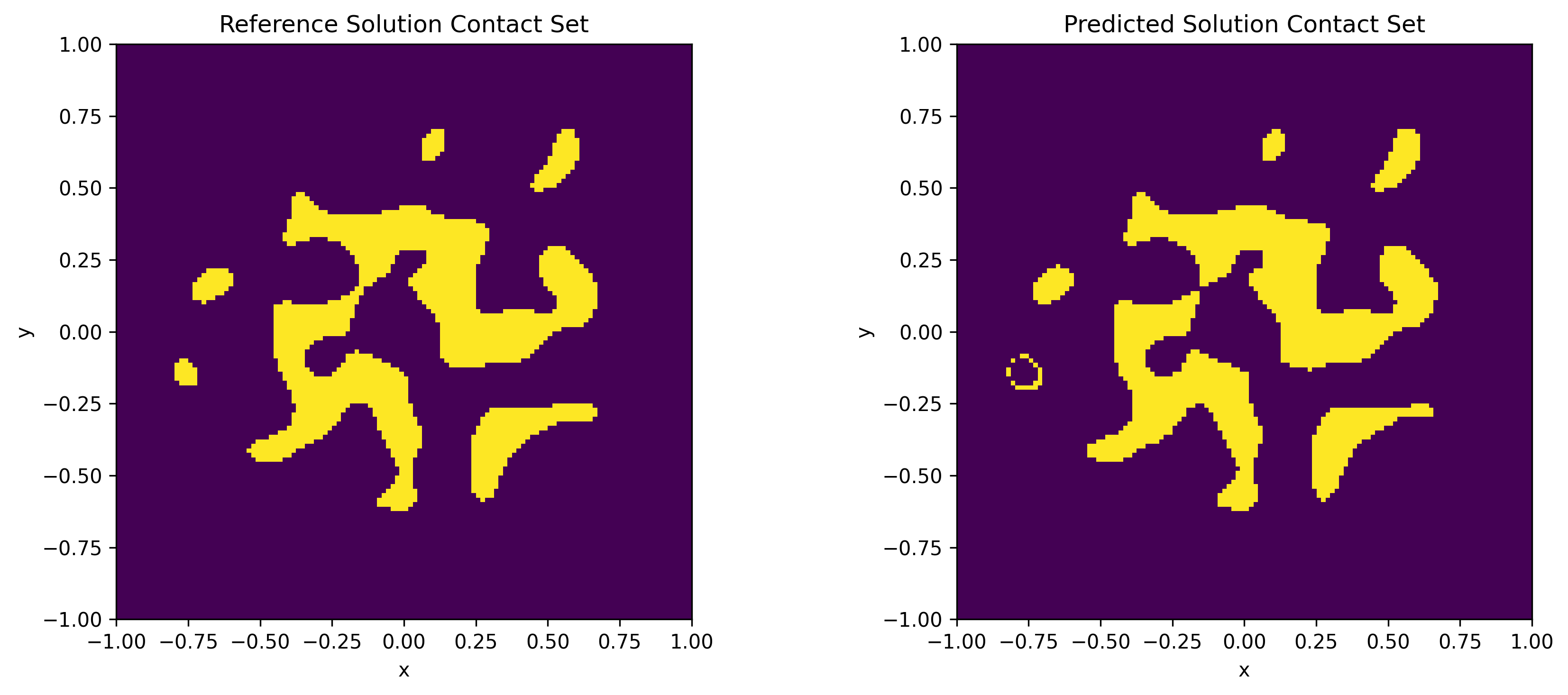}
    \caption{Comparison of the reference and predicted contact sets for the worst-case sample shown in Figure~\ref{exm2_worst}. The reference contact set is shown on the left and the predicted contact set on the right; contact regions are highlighted in yellow. The contact-set IoU is $0.938409$.}
    \label{exm2_contact}
\end{figure}

Figures~\ref{exm2_worst}--\ref{exm2_contact} complement the aggregate metrics by showing where the remaining errors occur in the higher-amplitude regime. Even on the worst-case FNO-100 test sample, selected by the largest sample-wise relative \(L^2\) error, the predicted solution still captures the global shape of the reference solution well. The dominant errors remain localized rather than global, but they are more sharply concentrated near regions where the contact set changes topology or bends rapidly. The contact-set plots convey the same message from a geometric viewpoint: the predicted contact region matches the dominant branches and components, while the remaining discrepancies are concentrated on thin protrusions, small isolated components, and highly curved boundary segments. These regions are particularly challenging for data-driven surrogates, since identification of the contact set there is especially sensitive to small perturbations in the predicted field.

\section{Conclusion}

We introduced FNO-LS, a Fourier-neural-operator method with a post-training least-squares readout refit for learning random obstacle-to-solution maps. After end-to-end FNO training, the nonlinear backbone is frozen and the final affine readout is recomputed by solving a linear least-squares problem over the training data. This gives the empirical squared-error optimal readout for the learned frozen features at low additional cost, without modifying the nonlinear representation.

Numerical experiments on two finite-band self-affine random obstacle ensembles show that FNO-LS gives the strongest overall performance among the tested models, especially in the higher-amplitude regime where the contact geometry is more intricate. The contact-set IoU and obstacle-violation metrics indicate that the improvement is not limited to bulk field accuracy, but also concerns structural features specific to obstacle problems.

Future work will consider jointly random coefficients and obstacles, nonrectangular geometries, and architectures or loss functions that enforce the unilateral constraint more explicitly.


\bigskip
\begin{thebibliography}{99}

\bibitem{Alphonse2024EVI}
A. Alphonse, M. Hinterm\"uller, A. Kister, C. H. Lun, and C. Sirotenko, A neural network approach to learning solutions of a class of elliptic variational inequalities, arXiv: 2411.18565, 2024.

\bibitem{Atkinson2009Theoretical}
K. Atkinson and W. Han, \emph{Theoretical Numerical Analysis: A Functional Analysis Framework}, Third Edition, Springer, New York, 2009.

\bibitem{Babuska2002}
I. Babu\v{s}ka and P. Chatzipantelidis, On solving elliptic stochastic partial differential equations, \emph{Computer Methods in Applied Mechanics and Engineering}, \textbf{191} (2002), 4093--4122.

\bibitem{Babuska2004}
I. Babu\v{s}ka, R. Tempone, and G. E. Zouraris, Galerkin finite element approximations of stochastic elliptic partial differential equations, \emph{SIAM Journal on Numerical Analysis}, \textbf{42} (2004), 800--825.

\bibitem{Babuska2007}
I. Babu\v{s}ka, F. Nobile, and R. Tempone, A stochastic collocation method for elliptic partial differential equations with random input data, \emph{SIAM Journal on Numerical Analysis}, \textbf{45} (2007), 1005--1034.

\bibitem{ElBahja2025PINNObstacle}
H. El Bahja, J. C. Hauffen, P. Jung, B. Bah, and I. Karambal, A physics-informed neural network framework for modeling obstacle-related equations, \emph{Nonlinear Dynamics}, {\textbf 113} (2025), 12533--12544. 

\bibitem{Bierig2015}
C. Bierig and A. Chernov, Convergence analysis of multilevel Monte Carlo variance estimators and application for random obstacle problems, \emph{Numerische Mathematik}, \textbf{130} (2015), 579--613.

\bibitem{Cheng2023ObstacleDNN}
X. Cheng, X. Shen, X. Wang, and K. Liang, A deep neural network-based method for solving obstacle problems, \emph{Nonlinear Analysis: Real World Applications}, \textbf{72} (2023), 103864.

\bibitem{Darehmiraki2022Obstacle}
M. Darehmiraki, A deep learning approach for the obstacle problem, In \emph{Proceedings of Academia-Industry Consortium for Data Science: AICDS 2020}, Springer Nature, Singapore, 2022, 179--188.

\bibitem{DL1976}
G. Duvaut and J.-L. Lions, \emph{Inequalities in Mechanics and Physics}, Springer-Verlag, Berlin--New York, 1976.

\bibitem{EJK2005}
C. Eck, J. Jaru\v{s}ek, and M. Krbec, \emph{Unilateral Contact Problems: Variational Methods and Existence Theorems}, Monographs and Textbooks in Pure and Applied Mathematics, Vol.~270, CRC Press, Boca Raton, 2005.

\bibitem{EigelHeissSchutte2025}
M. Eigel, C. Hei\ss{}, and J. E. Sch\"utte, Multi-level neural networks for high-dimensional parametric obstacle problems, arXiv: 2504.05026, 2025.

\bibitem{Forster2010}
R. Forster and R. Kornhuber, A polynomial chaos approach to stochastic variational inequalities, \emph{Journal of Numerical Mathematics}, \textbf{18} (2010), 235--255.

\bibitem{Gao2025ProxPINNs}
Y. Gao, Y. Song, Z. Tan, H. Yue, and S. Zeng, Prox-PINNs: A Deep Learning Algorithmic Framework for Elliptic Variational Inequalities, arXiv: 2505.14430, 2025.

\bibitem{Glowinski1984}
R. Glowinski, \emph{Numerical Methods for Nonlinear Variational Problems}, Springer-Verlag, New York, 1984.

\bibitem{TLG1981}
R. Glowinski, J.-L. Lions, and R. Tr\'emoli\`eres, \emph{Numerical Analysis of Variational Inequalities}, North-Holland, Amsterdam, 1981.

\bibitem{HR2013}
W. Han and B. D. Reddy, \emph{Plasticity: Mathematical Theory and Numerical Analysis}, Second Edition, Springer, New York, 2013.

\bibitem{HS2002}
W. Han and M. Sofonea, \emph{Quasistatic Contact Problems in Viscoelasticity and Viscoplasticity}, Studies in Advanced Mathematics, Vol.~30, American Mathematical Society, Providence, RI; International Press, Somerville, MA, 2002.

\bibitem{Hlavacek1988}
I. Hlav\'a\v{c}ek, J. Haslinger, J. Ne\v{c}as, and J. Lov\'i\v{s}ek, \emph{Solution of Variational Inequalities in Mechanics}, Springer-Verlag, New York, 1988.

\bibitem{Hueber2005}
S. H\"ueber and B. I. Wohlmuth, A primal-dual active set strategy for non-linear multibody contact problems, \emph{Computer Methods in Applied Mechanics and Engineering}, \textbf{194} (2005), 3147--3166.

\bibitem{Karkkainen2003}
T. K{\"a}rkk{\"a}inen, K. Kunisch, and P. Tarvainen, Augmented Lagrangian active set methods for obstacle problems, \emph{Journal of Optimization Theory and Applications}, \textbf{119} (2003), 499--533.

\bibitem{KO1988}
N. Kikuchi and J. T. Oden, \emph{Contact Problems in Elasticity: A Study of Variational Inequalities and Finite Element Methods}, SIAM, Philadelphia, 1988.

\bibitem{Kinderlehrer2000}
D. Kinderlehrer and G. Stampacchia, \emph{An Introduction to Variational Inequalities and Their Applications}, SIAM, Philadelphia, 2000.

\bibitem{Kornhuber2014}
R. Kornhuber, C. Schwab, and M. W. Wolf, Multilevel Monte Carlo finite element methods for stochastic elliptic variational inequalities, \emph{SIAM Journal on Numerical Analysis}, \textbf{52} (2014), 1243--1268.

\bibitem{Kovachki2021FNOTheory}
N. B. Kovachki, S. Lanthaler, and S. Mishra, On universal approximation and error bounds for Fourier neural operators, \emph{Journal of Machine Learning Research}, \textbf{22} (2021), 1--76.

\bibitem{Kovachki2023NeuralOperator}
N. B. Kovachki, Z. Li, B. Liu, K. Azizzadenesheli, K. Bhattacharya, A. M. Stuart, and A. Anandkumar, Neural operator: Learning maps between function spaces with applications to PDEs, \emph{Journal of Machine Learning Research}, \textbf{24} (2023), 1--97.

\bibitem{LanthalerStuart2026}
S. Lanthaler and A. M. Stuart, The parametric complexity of operator learning, \emph{IMA Journal of Numerical Analysis}, \textbf{46} (2026), 647--712.

\bibitem{Lee2024}
S. Lee and Y. Shin, On the Training and Generalization of Deep Operator Networks, \emph{SIAM Journal on Scientific Computing}, \textbf{46} (2024), C273--C296.

\bibitem{Li2021FNO}
Z. Li, N. Kovachki, K. Azizzadenesheli, B. Liu, K. Bhattacharya, A. M. Stuart, and A. Anandkumar, Fourier neural operator for parametric partial differential equations, In \emph{International Conference on Learning Representations}, 2021.

\bibitem{Lu2021DeepONet}
L. Lu, P. Jin, G. Pang, Z. Zhang, and G. E. Karniadakis, Learning nonlinear operators via DeepONet based on the universal approximation theorem of operators, \emph{Nature Machine Intelligence}, \textbf{3} (2021), 218--229.

\bibitem{Lu2022Comparison}
L. Lu, X. Meng, S. Cai, Z. Mao, S. Goswami, Z. Zhang, and G. E. Karniadakis, A comprehensive and fair comparison of two neural operators (with practical extensions) based on FAIR data, \emph{Computer Methods in Applied Mechanics and Engineering}, \textbf{393} (2022), 114778.

\bibitem{Persson2006}
B. N. J. Persson, Contact mechanics for randomly rough surfaces, \emph{Surface Science Reports}, \textbf{61} (2006), 201--227.

\bibitem{Ro1987}
J. F. Rodrigues, \emph{Obstacle Problems in Mathematical Physics}, North-Holland, Amsterdam, 1987.

\bibitem{Schwab2006}
C. Schwab and R. A. Todor, Karhunen--Lo\`eve approximation of random fields by generalized fast multipole methods, \emph{Journal of Computational Physics}, \textbf{217} (2006), 100--122.

\bibitem{SchwabStein2022ProxNet}
C. Schwab and A. Stein, Deep solution operators for variational inequalities via proximal neural networks, \emph{Research in the Mathematical Sciences}, \textbf{9} (2022), 36.

\bibitem{Todor2007}
R. A. Todor and C. Schwab, Convergence rates for sparse chaos approximations of elliptic problems with stochastic coefficients, \emph{IMA Journal of Numerical Analysis}, \textbf{27} (2007), 232--261.

\bibitem{Wang2021PIDeepONet}
S. Wang, H. Wang, and P. Perdikaris, Learning the solution operator of parametric partial differential equations with physics-informed DeepONets, \emph{Science Advances}, \textbf{7} (2021), eabi8605.

\bibitem{wang2010discontinuous}
F. Wang, W. Han, and X. Cheng, Discontinuous Galerkin methods for solving elliptic variational inequalities, \emph{SIAM Journal on Numerical Analysis}, \textbf{48} (2010), 708--733.

\bibitem{wang2014discontinuous}
F. Wang, W. Han, and X. Cheng, Discontinuous Galerkin methods for solving a quasistatic contact problem, \emph{Numerische Mathematik}, \textbf{126} (2014), 771--800.

\bibitem{wang2018friction}
F. Wang and H. Wei, Virtual element method for simplified friction problem, \emph{Applied Mathematics Letters}, \textbf{85} (2018), 125--131.

\bibitem{wang2020obstacle}
F. Wang and H. Wei, Virtual element methods for the obstacle problem, \emph{IMA Journal of Numerical Analysis}, \textbf{40} (2020), 708--728.

\bibitem{Xiu2007}
D. Xiu, Efficient collocational approach for parametric uncertainty analysis, \emph{Communications in Computational Physics}, \textbf{2} (2007), 293--309.

\bibitem{Zhao2022Obstacle}
X. E. Zhao, W. Hao, and B. Hu, Two neural-network-based methods for solving elliptic obstacle problems, \emph{Chaos, Solitons \& Fractals}, \textbf{161} (2022), 112313.

\bibitem{Zhu2026Obstacle}
C. Zhu, F. Wang, and W. Han,  Numerical Analysis of Stochastic Elliptic Variational Inequalities of the First Kind, arXiv: 2604.25111, 2026.

\end{thebibliography}
\end{document}